\documentclass{article}

\usepackage[utf8]{inputenc}

\usepackage{amsmath,amsfonts,amssymb,amsthm,cases, hyperref}

\usepackage{color}
\usepackage{vmargin}
\usepackage{dsfont}
\usepackage{esint} 

\setmarginsrb{2cm}{1cm}{2cm}{3cm}{1cm}{1cm}{2cm}{2cm}

\newcommand{\e}{\varepsilon}
\newcommand{\R}{\mathbb{R}}

\renewcommand{\div}{\operatorname{div}}
\newcommand{\curl}{\operatorname{curl}}
\newcommand{\dist}{\operatorname{dist}}

\newcommand{\supp}{\operatorname{supp}}

\renewcommand{\leq}{\leqslant}

\renewcommand{\geq}{\geqslant}

\newtheorem{Theorem}{Theorem}[section]
\newtheorem{Definition}[Theorem]{Definition}

\newtheorem{Proposition}[Theorem]{Proposition}
\newtheorem{Lemma}[Theorem]{Lemma}

\theoremstyle{remark}
\newtheorem{Remark}[Theorem]{Remark}
\newtheorem{Example}[Theorem]{Example}

\numberwithin{equation}{section}

\title{Small moving rigid body into a viscous incompressible fluid} 
\author{Christophe Lacave \& Tak\'eo Takahashi}
\def\adrese{
\begin{description}
\item[C. Lacave:] Univ Paris Diderot, Sorbonne Paris Cit\'e, Institut de Math\'ematiques de Jussieu-Paris Rive Gauche, 
UMR 7586, CNRS, Sorbonne Universit\'es, UPMC Univ Paris 06, F-75013, Paris, France.\\
and\\
Univ. Grenoble Alpes, IF, F-38000 Grenoble, France\\
CNRS, IF, F-38000 Grenoble, France\\
Email: \texttt{christophe.lacave@imj-prg.fr}\\
Web page: \texttt{\url{https://www.imj-prg.fr/~christophe.lacave/}}
\item[T. Takahashi:] Inria, Villers-l\`es-Nancy, F-54600, France\\
and\\
Institut \'Elie Cartan de Lorraine, UMR 7502, Vandoeuvre-l\`es-Nancy, F-54506, France\\
Email: \texttt{takeo.takahashi@inria.fr}\\
Web page: \texttt{\url{http://iecl.univ-lorraine.fr/~Takeo.Takahashi/}}
\end{description}
}

\date{\today}
\begin{document}
\maketitle

\begin{abstract}
We consider a single disk moving under the influence of a 2D viscous fluid and we study the asymptotic as the size of the solid tends to zero.

If the density of the solid is independent of $\varepsilon$, the energy equality is not sufficient to obtain a uniform estimate for the solid velocity. This will be achieved thanks to the optimal $L^p-L^q$ decay estimates of the semigroup associated to the fluid-rigid body system and to a fixed point argument. Next, we will deduce the convergence to the solution of the Navier-Stokes equations in $\R^2$.
\end{abstract}


\section{Introduction}

We study in this paper the asymptotic of a fluid-solid system as the solid is a rigid disk which shrinks to a point.
We first describe the fluid-solid system.

\subsection{The fluid-solid system}

We consider a rigid disk 
$$\mathcal{S}^\varepsilon(t)= \overline{B(h^\varepsilon(t),\varepsilon)}$$
immersed into a viscous incompressible fluid. At time $t>0$, the lagrangian coordinates to the body read 
\begin{equation*}
\eta (t,x):= h^\varepsilon (t)+R_{\theta^\varepsilon (t)}\left(x-h_{0}\right),
\end{equation*}
where for all $t$,
$$
h^\varepsilon(t)\in \mathbb{R}^2 \text{ with } h^\varepsilon(0)=h_{0}, \quad \theta^\varepsilon(t)\in \mathbb{R}, \quad 
R_\theta=\begin{pmatrix}
\cos \theta & -\sin\theta \\ \sin \theta & \cos \theta
\end{pmatrix}.
$$
 The domain of the fluid evolves through the formula
\begin{equation}\label{domaint}
\mathcal{F}^\varepsilon(t)
:=\R^2 \setminus \mathcal{S}^\varepsilon (t).
\end{equation}
We denote by $n:=n^\varepsilon(t,x)$ the exterior unit normal of $\partial \mathcal{F}^\varepsilon(t)$. 
The equations for the fluid-solid system read
\begin{align}
\frac{\partial u^{\varepsilon}}{\partial t}+ (u^{\varepsilon}\cdot\nabla) u^{\varepsilon}-\div \sigma(u^{\varepsilon},p^{\varepsilon})=0 &\quad t>0, \ x\in \mathcal{F}^{\varepsilon}(t), \label{eq01} 
\\
\div u^{\varepsilon} = 0 &\quad t>0, \ x\in \mathcal{F}^{\varepsilon}(t), \label{cou17}\\
\lim_{|x|\to \infty} u^{\varepsilon}(t,x)=0 & \quad t>0, \label{bd cond}\\
u^{\varepsilon}=(h^{\varepsilon})'(t)+ (\theta^{\varepsilon})'(t) (x-h^{\varepsilon}(t))^\perp & \quad t>0, \ x\in \partial \mathcal{S}^{\varepsilon}(t),\\
m^{\varepsilon} (h^\varepsilon)''(t)=-\int_{\partial \mathcal{S}^{\varepsilon}(t)} \sigma(u^{\varepsilon},p^{\varepsilon})n~d\gamma
& \quad t>0,\\
 J^{\varepsilon} (\theta^{\varepsilon})''(t)=-\int_{\partial \mathcal{S}^{\varepsilon}(t)} (x-h^{\varepsilon})^\perp \cdot \sigma(u^{\varepsilon},p^{\varepsilon})n~d\gamma & \quad t>0,\\
 u^{\varepsilon}(0,\cdot)=u_0^\varepsilon & \quad \text{in } \mathcal{F}^{\varepsilon}_0,\\
 h^\varepsilon(0)=h_{0},\quad (h^\varepsilon)'(0)=\ell_{0}^\varepsilon,\quad
 \theta^\varepsilon(0)=0, \quad & (\theta^\varepsilon)'(0)=r_{0}^\varepsilon.
 \label{fsr07}
\end{align}
Here and in what follows
\begin{equation*}
\sigma(u,p)=2\nu D(u)-p I_2,
\end{equation*}
with $\nu>0$ is the constant viscosity and
$$
D(u):=\frac 12 \left((\nabla u) + (\nabla u)^* \right).
$$
We write for any $x \in \mathbb{R}^2$,
\begin{equation*}
x^\perp:=\begin{pmatrix}
-x_2 \\ x_1 
\end{pmatrix}=R_{\pi/2} x.
\end{equation*}

It is convenient to extend the velocity field $u^{\varepsilon}$ inside the rigid disk as follows:
\begin{equation}\label{extension}
 u^{\varepsilon}(t,x)=(h^{\varepsilon})'(t)+ (\theta^{\varepsilon})'(t) (x-h^{\varepsilon}(t))^\perp 
 \quad t>0, \ x\in \mathcal{S}^{\varepsilon}(t).
\end{equation}
 
To apply the result in \cite{EHL}, we also need to assume that the center of the mass corresponds with the center of the disk. For simplicity, let us assume that the density $\rho^\varepsilon>0$ is constant in the disk. We define a global density in $\R^2$ by
 $$
 \rho^\varepsilon(t,x)=\left\{\begin{array}{ll}
1 & x\in \mathcal{F}^{\varepsilon}(t),\\
\rho^\varepsilon & x \in \mathcal{S}^{\varepsilon}(t).
\end{array}
 \right.
 \quad t\geq 0.
 $$

For any smooth open set $\mathcal{O}$, we define
\begin{itemize}
\item $V (\mathcal{O}):= \Bigl\{ \varphi \in C_0^\infty(\mathcal{O}) \: \vert \: \div \varphi = 0 \text{ in } \mathcal{O} \Bigl\}$;
\item $\mathcal{H}(\mathcal{O})$ the closure of $V (\mathcal{O})$ in the norm $L^2$:
\[
\mathcal{H}(\mathcal{O}) = \Bigl\{ \varphi \in L^2(\mathcal{O}) \: \vert \: \div \varphi = 0 \text{ in } \mathcal{O},\: \varphi\cdot n =0 \text{ at } \partial \mathcal{O} \Bigl\} ;
\]
\item $\mathcal{V}(\mathcal{O})$ the closure of $V (\mathcal{O})$ in the norm $H^1$:
\[
\mathcal{V}(\mathcal{O}) = \Bigl\{ \varphi \in H^1_{0}(\mathcal{O}) \: \vert \: \div \varphi = 0 \text{ in } \mathcal{O}\Bigl\}
\]
and its dual space by $\mathcal{V}'(\mathcal{O})$ with respect to $\mathcal{H}(\mathcal{O})$.
\end{itemize}
We also define $\mathcal{V}_R(\mathcal{F}^{\varepsilon}(t))$ the subspace of $\mathcal{V}(\R^2)$ of velocity fields that are rigid in the solid:
\begin{equation}\label{fsr09}
\mathcal{V}_R(\mathcal{F}^{\varepsilon}(t)):=\left\{ \varphi \in H^1(\R^2) \ ; \ D(\varphi)=0 \ \text{in} \ \mathcal{S}^{\varepsilon}(t) , \quad \div \varphi = 0 \right\}.
\end{equation}

Under the following hypotheses on the initial conditions
\begin{equation}\label{compatibility}
\begin{split}
 u^{\varepsilon}_0\in L^2(\mathcal{F}^\varepsilon_{0}), \quad \div u^{\varepsilon}_0=0, \quad
u^{\varepsilon}_0\cdot n= \ell_{0}^\varepsilon \cdot n \ 
\text{on} \ \partial \mathcal{S}^\varepsilon_{0},
\end{split}
\end{equation}
there exists a unique global weak solution $(u^{\varepsilon},\ h^\varepsilon, \ \theta^\varepsilon)$ see \cite{Takahashi&Tucsnak04}, in the sense of the definition below.

\begin{Definition}\label{D01}
We say that $(u^{\varepsilon},\ h^\varepsilon,\ \theta^\varepsilon)$ is a global weak solution of \eqref{eq01}--\eqref{fsr07} if, for any $T>0$, we have
\begin{equation*}
u^{\varepsilon} \in L^\infty(0,T;L^2(\R^2))\cap L^2(0,T;H^1(\R^2)), \quad u^{\varepsilon}(t,\cdot) \in \mathcal{V}_R(\mathcal{F}^{\varepsilon}(t)),
\end{equation*}
and if it satisfies the weak formulation
\begin{equation}\label{fsr08}
-\int_0^T \int_{\R^2} \rho^{\varepsilon} u^{\varepsilon} \cdot \left(\frac{\partial \varphi^{\varepsilon}}{\partial t} + (u^{\varepsilon} \cdot \nabla)\varphi^{\varepsilon} \right) \, dx ds
+ 2\nu \int_0^T \int_{\R^2} D(u^{\varepsilon}): D(\varphi^{\varepsilon}) \, dxds = \int_{\R^2} \rho^{\varepsilon} u^{\varepsilon}_0(x) \cdot \varphi^{\varepsilon}(0,x) \, dx,
\end{equation}
for any $\varphi^\varepsilon \in C^1_c([0,T);H^1(\R^2))$ such that $\varphi^{\varepsilon}(t,\cdot) \in \mathcal{V}_R(\mathcal{F}^{\varepsilon}(t))$.
\end{Definition}

\subsection{Massless pointwise particle in the whole plane}

When $\varepsilon\to 0$, we establish the convergence of $u^\varepsilon$ to the unique solution of the Navier-Stokes equations in the whole plane $\R^2$.

The asymptotic behavior of the fluid motion around shrinking obstacles is already considered in several recent papers. Iftimie, Lopes Filho and Nussenzveig Lopes \cite{IftLopNus2006a} have studied the case of one small fixed obstacle in an incompressible viscous fluid in 2D. Iftimie and Kelliher \cite{IftKel2009a} have treated the same situation in 3D. In \cite{Lac2009c, Lac2015b} Lacave has considered the case of one thin obstacle shrinking to a curve in 2D and 3D.

There is also a large literature about porous medium in the homogenization framework. Since the pioneer work of Cioranescu and Murat \cite{CM82} for the Laplace problem, the Navier-Stokes system was studied, in particular, by Allaire \cite{Allaire90a,Allaire90b}. We also mention \cite{conca1,conca2,Mikelic91,SP1,SP2,Tartar80} for the fluid motion through a perforated domain.

In all the above studies, the general strategy relies on energy estimate to get a uniform estimate in $H^1$. It turns out that such an estimate is sufficient to pass to the limit in the weak formulation by a troncature procedure. Namely, for a test function $\varphi\in \mathcal{D}(\Omega)$ and for a cutoff function $\chi^\e$, we note that $\chi^\e \varphi$ is an admissible test function for the Laplace problem in the perforated domain $\Omega^\varepsilon$. If the inclusions are far enough, for standard cutoff function, $\| \chi^\e\|_{W^{1,p}}$ remains bounded only for $p\leq 2$ in dimension two, which allows to pass to the limit in terms like $\int \nabla u^\e : \nabla(\chi^\e \varphi)$. For the Navier-Stokes equations, the cutoff procedure is more complicated and relies on Bogovski{\u\i} operators. Indeed, we need approximated test functions that are divergence free (see \cite{Allaire90a,Lac2015b} and Section~\ref{sec:cutoff}).

When the obstacles can move under the influence of the fluid, we also need to control uniformly the velocities of the solids. 
In the case of the system \eqref{eq01}--\eqref{fsr07} or more generally in the case of a system with several rigid bodies moving into a viscous incompressible fluid, this control can be obtained from the energy estimate if the masses are independent of $\varepsilon$ (see Section~\ref{sec:massive} for details). 

If the masses tend to zero, it is no more possible to deduce estimates of the velocities of the solids independently of $\varepsilon$ from the energy estimate.
One could try to get an estimate of rigid velocities from the boundary condition. However, since the size of the solids tend to zero, this leads to look for a $C^0$-estimate for the fluid velocity, and thus for $H^{s}$ estimates with $s> 1$. It was the strategy followed in \cite{DasRob2011b,SilTak2014b} with a $H^2$ analysis. Unfortunately, these articles are based on uniform elliptic estimates in the exterior of a small obstacle which fail for $s>1$ (see a counter-example related to these estimates in \cite{CPRXpreprint}). 

Our strategy is different here. Our basic remark is that the small obstacle problem is related to the long-time behavior though the scaling property of the Navier-Stokes equations $u^{\varepsilon}(t,x)=\varepsilon^{-1} u^1(\varepsilon^{-2} t, \varepsilon^{-1} x)$. For one disk moving in the plane, the long-time behavior has been recently studied by Ervedoza, Hillairet and Lacave in \cite{EHL}. In particular, they have obtained the optimal decay estimates of the Stokes semigroup, i.e. with the rates corresponding to the heat kernel (which are invariant to the parabolic scaling). These estimates are the key to treat the massless pointwise particle.

The goal of the main theorem is to treat the case where the disk shrinks to a {\it massless pointwise particle}:
\begin{equation}\label{massless}
 \rho^\varepsilon =\rho_{0}
\end{equation}
hence,
\begin{equation}\label{massless2}
m^\varepsilon = \varepsilon^2 m^1 \quad \text{and} \quad J^\varepsilon= \varepsilon^{4} J^1.
\end{equation}

We consider the massless case for small data:
\begin{Theorem}\label{mainthm2}
Assume \eqref{massless}. Then there exists $\lambda_0$ such that the following holds.

Let $(u_0^{\varepsilon}, \ell_{0}^\varepsilon,r_{0}^\varepsilon)$ be a family in $L^2(\mathcal{F}^\varepsilon_{0})\times \R^2 \times \R$ verifying \eqref{compatibility} and such that
\begin{equation}\label{petitesse}
 \varepsilon |\ell_{0}^{\varepsilon}|,\ \varepsilon^{2} |r_{0}^{\varepsilon}|,\ \| u_0^{\varepsilon} \|_{L^2(\mathcal{F}^\varepsilon_{0})} \leq \lambda_0
\end{equation}
and
\begin{equation}\label{fsr28}
u_0^{\varepsilon} \rightharpoonup u_0 \quad \text{in} \quad L^2(\R^2).
\end{equation}
Then for any $T>0$ we have
\begin{equation}\label{fsr152}
u^{\varepsilon} \overset{*}{\rightharpoonup} u \quad \quad \text{ in } \quad L^\infty(0,T;L^2(\R^2))\cap L^2(0,T;H^1(\R^2)) 
\end{equation}
where $u$ is the weak solution of the Navier-Stokes equations in $\R^2$ associated to $u_0$: for any $\varphi\in C^1_c([0,T);\mathcal{V}(\R^2)),$
\begin{equation*}
-\int_0^T \int_{\R^2} u \cdot \left(\frac{\partial \varphi}{\partial t} + (u \cdot \nabla)\varphi \right) \, dx ds
+\nu \int_0^T \int_{\R^2} \nabla u: \nabla \varphi \, dx ds= \int_{\R^2} u_0(x) \cdot \varphi(0,x) \, dx.
\end{equation*}
\end{Theorem}

\begin{Remark}
For any $T>0$, we will actually establish in Section~\ref{sec:conv} that, up to a subsequence, we have
\begin{equation*}
h^{\varepsilon} \to h \quad \text{uniformly in} \ [0,T],
\end{equation*}
and in Section~\ref{sec:limit} that, for any $\mathcal{O} \Subset \R^2 \setminus \{ h(t) \}$ (for all $t\in (t_{1},t_{2})$ with some $t_{1},t_{2}\in [0,T]$), we have
\[
\mathbb{P}_{\mathcal{O}} u^\varepsilon \to \mathbb{P}_{\mathcal{O}} u \text{ strongly in } L^2(t_{1},t_{2};L^4(\mathcal{O})),
\]
where $\mathbb{P}_{\mathcal{O}}$ is the Leray projector. This strong limit will be used to pass to the limit in the non-linear term.

As we recover at the limit the weak solution of the Navier-Stokes equations in the whole plane, and as this solution is unique by the Leray theorem, we will deduce that we do not need to extract a subsequence in \eqref{fsr152}.
\end{Remark}

For a 2D ideal incompressible fluid governed by the Euler equations, the case of a massive pointwise particle (i.e. where $m^\varepsilon=m^1$ is independent of $\varepsilon$) in the whole plane was treated in \cite{GLS1}, a massless pointwise particle in the whole plane in \cite{GLS2} and both case in a bounded domain in \cite{GMS}. In these works, non-trivial limit was obtained (namely, Kutta-Joukowski lift force or vortex-wave system) when we consider non-zero initial circulations around the small solids.

\subsection{Plan of the paper}

The remainder of this work is organized in four sections.

In the next section, we provide three examples where the initial convergence \eqref{fsr28} holds.

We establish in Section~\ref{sec:est} some uniform estimates on $(u^\varepsilon,(h^\varepsilon)')$. The energy estimate will give us directly a good estimate for the fluid velocity $u^\varepsilon$ but not for the disk velocity $(h^\varepsilon)'$. Thanks to the results of \cite{EHL}, we will prove that some $L^p-L^q$ estimates of the Stokes semigroup are independent of $\varepsilon$, and by a fixed point argument we will get a uniform estimate of the disk velocity.

Section~\ref{sec:3} is dedicated to the passing to the limit. We introduce the cutoff procedure which follows the trajectory of the solid. A crucial point is to construct a corrected test function $\varphi^\eta$ which satisfies the divergence free condition. This will be obtained by the Bogovski{\u\i} operator \cite{Bogovskii79,Bogovskii80}. Then we follow the analysis developed in \cite{Lac2015b}. Roughly, we pass first to the limit $\varepsilon\to 0$ far away from the solid to get that $u$ satisfies the Navier-Stokes equations in this region. Next, we pass to the limit $\eta\to 0$ in the cutoff function, to prove that the equations are also verified in the vicinity of the massless pointwise particle.

We will discuss in the last section the three dimensional case and we will give the extension of our main result in the case where several solids (with any shapes) tend to massive pointwise particles in bounded domains. In this case the energy estimate is sufficient to obtain a uniform estimate of the solid velocities, which allows us to reach more general geometric configurations than in the massless case.

\section{Examples of initial conditions}

In this paragraph, we develop three examples of family $(u_{0}^\varepsilon, \ell_{0}^\varepsilon)_{\varepsilon}$ satisfying the compatibility condition \eqref{compatibility} which converges in $L^2(\R^2)$.

\begin{Example}
The first trivial example is the case where $u_{0}^\varepsilon$ is independent of $\varepsilon$. Namely, let us consider $(u_{0}^{\varepsilon_{0}}, \ell_{0}^{\varepsilon_{0}},r_{0}^{\varepsilon_{0}})$ satisfying \eqref{compatibility} in $\mathcal{F}^{\varepsilon_{0}}_{0}$ for some $\varepsilon_{0}>0$. The extension of $u_0^{\varepsilon_{0}}$ by $\ell_{0}^{\varepsilon_{0}} + r_{0}^{\varepsilon_{0}} (x-h_{0})^\perp$ inside the disk verifies also \eqref{compatibility} in $\mathcal{F}^{\varepsilon}_{0}$ for any $\varepsilon\in (0,\varepsilon_{0}]$.
\end{Example}

\begin{Example}
In domains depending on $\varepsilon$, a standard setting is to give an initial data in terms of an independent vorticity $\omega_{0}=\curl u_{0}^\varepsilon$ and initial circulation $\gamma_{0}=\oint_{\partial B(h_{0},\varepsilon)} u_{0}^\varepsilon \cdot \tau \, ds$ (see, e.g., \cite{GLS1,GLS2,IftKel2009a,IftLopNus2006a,Lac2009c,Lac2015b}). For the 2D Euler equations, the vorticity is the natural quantity because it satisfies a transport equation, which implies some conserved properties (e.g. the $L^p$ norm of the vorticity for $p\in [1,\infty]$ and the circulation of the velocity).

For the 2D Navier-Stokes equations in domains with fixed boundaries, the vorticity and the circulation are less relevant, because the Dirichlet boundary condition implies that the circulation is zero for $t>0$, and the vorticity equation does not give anymore the conservation of the $L^p$ norm. For these equations, the standard framework is related to the energy estimate \eqref{fsr10}, i.e. to consider initial data belonging to $L^2 (\mathcal{F}^\e_{0})$. In terms of the vorticity, we recall that $u_{0}^\e(x) = (\gamma_{0}+ \int_{\mathcal{F}^\e_{0}} \omega_{0})\frac{x^\perp}{2\pi |x|^2} + \mathcal{O}(\frac1{|x|^2})$ at infinity (see for instance \cite[Section 2]{GLS1} or \cite[Section 3]{IftLopNus2006a}), hence
\[
u_{0}^\e \in L^2 (\mathcal{F}^\e_{0}) \Longleftrightarrow \gamma_{0}+ \int_{\mathcal{F}^\e_{0}} \omega_{0}=0.
\]
For these reasons, the most natural condition is $\gamma_{0}=\int \omega_{0} =0$. In this case, it is easy to prove the following result.
\end{Example}

\begin{Lemma}
Let $(\ell_{0},r_{0})\in \mathbb{R}^3$, $\omega_0\in L^\infty_{c}(\R^2\setminus\{ 0\})$ fixed such that $\int \omega_{0}=0$. Then, for $\e$ small enough such that ${\rm supp \ }\omega_{0}\cap B(0,\e)=\emptyset$, we have a unique solution $u_{0}^\e$ in $L^2(\mathcal{F}^\e_0)$ of 
\begin{gather*}
 \div u^\e_0=0 \text{ in } \mathcal{F}^\e_0, \quad \curl u^\e_0 = \omega_0 \text{ in } \mathcal{F}^\e_0, \quad \lim_{ |x| \to \infty} u_0^\e(x) =0 ,\\
u_0^\e \cdot n = \ell_{0} \cdot n \text{ on } \partial B(0,\e), \quad \oint_{\partial B(0,\e)} u_0^\e\cdot \tau \, ds=0 .
\end{gather*}
Moreover, extending $u_{0}^\e$ by $\ell_{0}+r_{0}x^\perp$ in $B(0,\e)$ we have
\begin{equation*}
 u_0^\e \rightharpoonup u_0 \quad \text{ weakly in } L^2(\R^2),
\end{equation*}
where $u_{0}=K_{\R^2}[\omega_{0}]=\frac{x^\perp}{2\pi|x|^2} * \omega_{0} $ is the unique vector field in $L^2(\R^2)$ such that
\[
\div u_0=0 \text{ in } \R^2, \quad \curl u_0 = \omega_0 \text{ in } \R^2, \quad \lim_{|x|\to \infty }u_0(x)=0.
\]
\end{Lemma}

\begin{proof} The existence and uniqueness of $u_{0}^\e$ is well-known (see e.g. \cite[Section 2]{GLS1}):
\[
u_{0}^\e(x)= \frac1{2\pi} \int_{B(0,\e)^c} \frac{(x-y)^\perp}{|x-y|^2} \omega_{0}(y)dy + \frac1{2\pi} \int_{B(0,\e)^c} \Big( \frac{x}{|x|^2}- \frac{x-\e^2 y^*}{|x- \e^2 y^* |^2}\Big)^\perp \omega_{0}(y)dy - \e^2 \sum_{j=1}^2 (\ell_{0})_{j} \nabla\Big( \frac{x_{j}}{|x|^2} \Big)
\]
with the notation $y^*=y/|y|^2$. By a standard computation, we note that
\[
\Big\| \e^2 \sum_{j=1}^2 (\ell_{0})_{j} \nabla\Big( \frac{x_{j}}{|x|^2} \Big) \Big\|_{L^2(\mathcal{F}^\e_{0})} \leq C \e | \ell_{0}|.
\]

It is also rather classical to prove that the second integral in the right hand side tends to zero as $\e\to 0$. For instance, the authors establish in \cite[Lemmas 7 and 10]{IftLopNus2006a} the uniform estimate in $L^2(\R^2)$ and the convergence to zero in $\mathcal{D}'(\R^2)$, which implies the weak convergence in $L^2(\R^2)$.
 
As $\frac1{2\pi} \int_{B(0,\e)^c} \frac{(x-y)^\perp}{|x-y|^2} \omega_{0}(y)dy = K_{\R^2}[\omega]=u_{0}$, this ends the proof.
\end{proof}

\begin{Remark}
Even if the zero circulation condition is mandatory for strong solutions to the Navier-Stokes equations in fixed domain, for the fluid-solid system the no-slip boundary condition would imply
$$\oint_{\partial B(h_{0},\varepsilon)} u_{0}^\varepsilon \cdot \tau \, ds = 2\pi \varepsilon^2 r_{0}^\varepsilon.$$
Hence, an interesting extension could be to study the case of non zero initial circulation $\gamma_{0}^\varepsilon$. If we assume that $\gamma_{0}^\varepsilon$ is independent of $\varepsilon$, some singular terms appear at the limit of the form $\gamma_{0}\frac{(x-h_{0})^\perp}{2\pi |x- h_{0}|^2}$, which does not belong to $L^2_{\rm loc}(\R^2)$, but only to $L^p_{\rm loc}(\R^2)$ for all $p\in [1,2)$. Another difficulty in this case is to ensure that $\gamma_{0}^\varepsilon+ \int_{\mathcal{F}^\e_{0}} \omega_{0}=0$ for any $\varepsilon$, in order to state that the initial velocity is square integrable at infinity. A possibility could be to chose $\gamma_{0}^\varepsilon= \int_{B(h_{0},\varepsilon)}\omega_{0}$, i.e. $r_{0}^\varepsilon=\fint_{B(h_{0},\varepsilon)}\omega_{0}$. Without this condition, $u^\varepsilon_{0}$ belongs only to $L^p(\mathcal{F}_{0}^\varepsilon)$ for all $p\in (2,\infty]$. 

Therefore, a circulation independent of $\varepsilon$ and a vorticity with non zero mean value would require to work in the Marcinkiewicz space $L^{2,\infty}(\mathcal{F}_{0}^\varepsilon)$ (weak $L^2$ space). Even if this space is less classical that $L^2$, there is a large literature for the well-posedness of the Navier-Stokes equations in fixed domain, because it corresponds to relevant initial data for the Euler equations, and also because the self-similar solutions (as the Lamb-Oseen vortex) in the whole plane belongs to $L^{2,\infty}$. In this case, the Cauchy theory is well-known, and Iftimie, Lopes Filho and Nussenzveig Lopes managed in \cite{IftLopNus2006a} to consider the small obstacle problem with non-zero initial circulation and initial vorticity with non-zero mean value. For the fluid-solid problem, such a Cauchy theory is not yet established. As the optimal decay estimates for the Stokes semigroup are now known in $L^p-L^q$ \cite{EHL}, we guess that it would be possible to extend it by interpolation to Marcinkiewicz spaces, and then to prove a well-posedness result for the full fluid-solid system. Such an extension would require more work and could be interesting, but the main goal of this article is to stay in the standard framework for the fluid solid problem and to treat the same question as \cite{DasRob2011b,SilTak2014b}.
\end{Remark}

\begin{Example}
Another example of initial conditions satisfying \eqref{compatibility} can be obtained by truncating a stream function associated to a vector 
field $u_{0}$ defined on $\R^2$.

Let us consider $u_{0}=\nabla^\perp \psi_{0} \in L^2(\R^2)$ such that
\[
\div u_0=0 \quad \text{in} \ \R^2, \quad \curl u_0 \in L^1\cap L^q(\R^2) \quad \text{with} \ q>1, \quad \lim_{|x|\to \infty}u_0(x) = 0 .
\]
We denote by $\chi$ a smooth cutoff function such that $\chi(x)\equiv 0$ in $B(0,3/2)$ and $\chi(x)\equiv 1$ in $B(0,2)^c$. We consider $(\ell_{0},r_{0})\in \R^3$ given, then we define
\[
u_{0}^\varepsilon := \nabla^\perp \Bigg( \psi_{0}(x) \chi\left(\frac{x-h_{0}}\varepsilon \right) + \Big(1-\chi\left(\frac{x-h_{0}}\varepsilon \right)\Big) (\ell_{0}\cdot (x-h_{0})^\perp + r_{0}\tfrac{ |x-h_{0}|^2}2 ) \Bigg)
\] 
which is divergence free, tending to $0$ at infinity, equal to $u_{0}$ far away the solid, and equal to $\ell_{0}+r_{0}(x-h_{0})^\perp$ in the vicinity of $\mathcal{S}^\varepsilon_{0}$. 
By local elliptic regularity, we can show that $\psi_{0}\in L^\infty(B(h_{0},2))$ and as $\frac{1}{\varepsilon}\nabla \chi\left(\frac{\cdot-h_{0}}\varepsilon\right)$ converges weakly to $0$ in $L^2$,
one can check that $u_{0}^\e \rightharpoonup u_{0}$ in $L^2(\R^2)$.

Actually, as we consider only one solid, we can chose $\psi_{0}$ such that $\psi_{0}(h_{0})=0$ and in that case, we can prove that the convergence holds strongly in $L^2(\R^2)$.

A last cutoff example comes from the porous medium analysis or the thin obstacle problem (see \cite{Allaire90a, Lac2015b}). Instead of truncating the steam function, we truncate directly the vector field $u_{0}$ and we add a correction to restore the divergence free condition:
\[
u_{0}^\varepsilon := \chi\left(\frac{x-h_{0}}\varepsilon \right) u_{0}(x) + \Big(1-\chi\left(\frac{x-h_{0}}\varepsilon \right)\Big) (\ell_{0} + r_{0} (x-h_{0})^\perp ) + g^\varepsilon(x)
\]
where $g^\varepsilon \in H^1_{0}(B(h_{0},2\varepsilon)\cap \mathcal{F}_{0}^\varepsilon)$ satisfies
\[
\div g^\varepsilon=- \frac1\varepsilon \nabla \chi\left(\frac{x-h_{0}}\varepsilon \right) \cdot u_{0}(x) + \frac1\varepsilon \nabla\chi\left(\frac{x-h_{0}}\varepsilon \right) \cdot (\ell_{0} + r_{0} (x-h_{0})^\perp ) .
\]
Such a function can be constructed thanks to the Bogovski{\u\i} operator, and we can prove that
\[
\| g^\varepsilon \|_{L^2}\leq C \varepsilon \Big( \|u_{0}\|_{L^2} + \varepsilon |\ell_{0}| + \varepsilon^2 |r_{0}| \Big).
\]
See later Proposition~\ref{lem:bogo} for details. This implies that $u_{0}^\varepsilon \to u_{0}$ in $L^2(\R^2)$.
\end{Example}

\section{Uniform estimates}\label{sec:est}

Let $(u_0^{\varepsilon}, \ell_{0}^\varepsilon,r_{0}^\varepsilon)$ be a family in $L^2(\mathcal{F}^\varepsilon_{0})\times \R^2 \times \R$ verifying \eqref{compatibility} and such that, up to the extension \eqref{extension}, $u_{0}^\varepsilon$ is bounded in $L^2(\R^2)$ (see \eqref{fsr28}).

\subsection{Change of variables and energy estimate}\label{sec:411}

As in \cite{EHL,Takahashi&Tucsnak04}, we make the change of variables 
$$
v^{\varepsilon}(t,x)=u^{\varepsilon}(t,x-h^{\varepsilon}(t)), \quad 
q^{\varepsilon}(t,x)=p^{\varepsilon}(t,x-h^{\varepsilon}(t)), 
$$
and we define 
$$
\ell^{\varepsilon}(t)=(h^{\varepsilon})'(t), \quad r^{\varepsilon}(t)=(\theta^{\varepsilon})'(t).
$$

The vector fields $v^{\varepsilon}$ is the weak solution of a system similar to \eqref{eq01}--\eqref{fsr07}:
{\allowdisplaybreaks
\begin{align}
\frac{\partial v^{\varepsilon}}{\partial t}+ ([v^{\varepsilon}-\ell^\varepsilon]\cdot\nabla) v^{\varepsilon}-\div \sigma(v^{\varepsilon},q^{\varepsilon})
=0 &\quad t>0, \ x\in \mathcal{F}^{\varepsilon}_0, 
\label{mle0.4} 
\\
\div v^{\varepsilon} = 0 &\quad t>0, \ x\in \mathcal{F}^{\varepsilon}_0,\\
\lim_{|x|\to \infty} v^{\varepsilon}(x)=0 & \quad t>0,\\
v^{\varepsilon}(t,x)=\ell^{\varepsilon}(t)+ r^{\varepsilon}(t) x^\perp & \quad t>0, \ x\in \partial \mathcal{S}_0^{\varepsilon},\\
m^{\varepsilon} (\ell^\varepsilon)'(t)=-\int_{\partial \mathcal{S}_0^{\varepsilon}} \sigma(v^{\varepsilon},q^{\varepsilon})n~d\gamma
& \quad t>0,\\
 J^{\varepsilon} (r^{\varepsilon})'(t)=-\int_{\partial \mathcal{S}_0^{\varepsilon}(t)} x^\perp \cdot \sigma(v^{\varepsilon},q^{\varepsilon})n~d\gamma & \quad t>0,\\
 v^{\varepsilon}(0,\cdot)=v_0^\varepsilon & \quad \text{in } \mathcal{F}^{\varepsilon}_0,\\
 \ell^\varepsilon(0)=\ell_{0}^\varepsilon,\quad
 r^\varepsilon(0)=r_{0}^\varepsilon.
 \label{mle0.5}
\end{align}
}

We set the global density in $\mathbb{R}^2$:
 $$
 \rho^\varepsilon(x)=\left\{\begin{array}{ll}
1 & x\in \mathcal{F}^{\varepsilon}_0,\\
\rho & x \in \mathcal{S}^{\varepsilon}_0.
\end{array}
 \right.
 $$

We can define a weak solution 
\begin{Definition}\label{D02}
We say that $(v^{\varepsilon},\ \ell^\varepsilon,\ r^\varepsilon)$ is a global weak solution of \eqref{mle0.4}--\eqref{mle0.5} if, for any $T>0$, we have
\begin{equation*}
v^{\varepsilon} \in L^\infty(0,T;L^2(\mathbb{R}^2))\cap L^2(0,T;H^1(\mathbb{R}^2)), \quad 
v^{\varepsilon}(t,x)=\ell^{\varepsilon}(t)+ r^{\varepsilon}(t) x^\perp \quad \text{in } \mathcal{S}^{\varepsilon}_0,
\end{equation*}
if it satisfies the weak formulation
\begin{equation*}
-\int_0^T \int_{\R^2} \rho^{\varepsilon} v^{\varepsilon} \cdot \left(\frac{\partial \varphi^{\varepsilon}}{\partial t} + ([v^{\varepsilon}-\ell^{\varepsilon}] \cdot \nabla)\varphi^{\varepsilon} \right) \, dx ds
+ 2\nu \int_0^T \int_{\R^2} D(v^{\varepsilon}): D(\varphi^{\varepsilon}) \, dxds = \int_{\R^2} \rho^{\varepsilon} v^{\varepsilon}_0(x) \cdot \varphi^{\varepsilon}(0,x) \, dx,
\end{equation*}
for any $\varphi^\varepsilon \in C^1_c([0,T);H^1(\mathbb{R}^2))$ such that $\varphi^{\varepsilon}(t,\cdot) \in \mathcal{V}_R(\mathcal{F}^{\varepsilon}_0)$.
\end{Definition}

One can check that $u^\varepsilon$ is a weak solution in the sense of Definition~\ref{D01} if and only if $v^\varepsilon$ is a weak solution in the sense of the above definition.

We also define the following functional spaces for $p\in [1,\infty]$,
$$
\mathcal{L}^p_{\varepsilon}=\left\{ v\in L^p(\mathbb{R}^2) \ ; \ \div v=0 \ \text{in} \ \mathbb{R}^2, \quad D(v)=0 \ \text{in}\ \mathcal{S}^{\varepsilon}_0 \right\},
$$
with the norm (for $p\neq \infty$)
$$
\| v \|_{\mathcal{L}^p_{\varepsilon}} = \left( \int_{\mathbb{R}^2} \rho^{\varepsilon} |v|^p \, dx \right)^{1/p}.
$$
For $p=\infty$, the norm $\mathcal{L}^{\infty}$ is the classical $L^{\infty}$ norm.
We recall that for any $v\in \mathcal{L}^p_{\varepsilon}$, there exists $(\ell_v,r_v)\in \mathbb{R}^3$ such that
\begin{equation*}
v(y)=\ell_v + r_v y^\perp \quad (y\in \mathcal{S}^{\varepsilon}_0).
\end{equation*}
Moreover, one can deduce $\ell_v$ from $v$ by
\begin{equation}\label{ellv}
\ell_{v} := \frac{1}{| \mathcal{S}^{\varepsilon}_0 |}\int_{\mathcal{S}^{\varepsilon}_0} v \, dy.
\end{equation}

One can write the system \eqref{mle0.4}--\eqref{mle0.5} in the following abstract form:
\begin{equation*}
\partial_{t} v^\varepsilon + A^\varepsilon v^\varepsilon = \mathbb{P}^\varepsilon \div F^\varepsilon(v^\varepsilon), \quad v^\varepsilon(0)=v_0^\varepsilon,
\end{equation*}
where 
\begin{equation*}
\mathcal{D}(\mathcal{A}^\varepsilon):=\left\{ v\in H^2(\mathbb{R}^2) \ ; \ \div v=0 \ \text{in} \ \mathbb{R}^2, \quad D(v)=0 \ \text{in}\ \mathcal{S}^{\varepsilon}_0 \right\},
\end{equation*}
\begin{equation*}
\mathcal{A}^\varepsilon v :=
\left\{
\begin{array}{ll}
-\nu \Delta v & \text{in}\ \mathcal{F}^{\varepsilon}_0,\\[0.2cm]
\displaystyle \frac{2\nu}{m^\varepsilon} \int_{\partial \mathcal{S}_0^\varepsilon} D(v)n \, ds + 
\frac{2\nu}{J^\varepsilon} \left(\int_{\partial \mathcal{S}_0^\varepsilon} y^\perp\cdot D(v)n \, dy\right) x^\perp & \text{in}\ \mathcal{S}^{\varepsilon}_0.
\end{array}
\right. \quad (v\in \mathcal{D}(\mathcal{A}^\varepsilon)),
\end{equation*}
\begin{equation*}
A^\varepsilon:= \mathbb{P}^\varepsilon \mathcal{A}^\varepsilon,
\end{equation*}
\begin{equation}
 \label{eq:Feps}
	F^\varepsilon (v^\varepsilon)= \left\{
	\begin{array}{rcll}
		v^\varepsilon\otimes (\ell_{v^\varepsilon}- v^\varepsilon) & \text{ on $\mathcal{F}_{0}^\varepsilon$}\,\\
		 0 &\text{ on $\mathcal{S}_{0}^\varepsilon$},
	 \end{array}\right.
\end{equation}
and where $\mathbb{P}^\varepsilon$ denotes the projector from $L^p(\R^2)$ to $\mathcal{L}^p_\varepsilon$, 
and $\ell_{v^\varepsilon}$ is defined through \eqref{ellv}. Note that in the definition of $\mathcal{A}^\varepsilon$, $D(v)n$ corresponds to the trace of the restriction of $D(v)$ to the fluid domain.

The operator $-A^\varepsilon$ is the infinitesimal generator of a semigroup of $(S^\varepsilon(t))_{t\geq 0}$ in $\mathcal{L}^p_\varepsilon$ for $p\in (1,\infty)$
(see \cite{EHL}).
Then, Duhamel's formula gives the following integral formulation of the above equations:
\begin{equation}\label{Duhamel}
	 v^\varepsilon(t) = S^\varepsilon(t) v_{0}^\varepsilon + \int_{0}^t S^\varepsilon(t-s) \mathbb{P}^\varepsilon \div F^\varepsilon(v^\varepsilon(s))\, ds.
\end{equation}

By Sobolev embedding, it is classical to deduce from the weak formulation (see Definition~\ref{D02}) that $\partial_{t} v^\varepsilon$ belongs to $L^2(0,T; \mathcal{V}_R'(\mathcal{F}^{\varepsilon}_{0}))$ and that the relation
\begin{equation*}
\int_0^T \int_{\R^2} \rho^{\varepsilon} \left( \partial_{t }v^{\varepsilon} \cdot \varphi^{\varepsilon} -v^{\varepsilon} \cdot ([v^{\varepsilon}-\ell^{\varepsilon}] \cdot \nabla\varphi^{\varepsilon}) \right) \, dx ds
+ 2\nu \int_0^T \int_{\R^2} D(v^{\varepsilon}): D(\varphi^{\varepsilon}) \, dxds = 0
\end{equation*}
is satisfied for any $\varphi^\varepsilon \in L^2(0,T;\mathcal{V}_R(\mathcal{F}^{\varepsilon}_{0}))$. In particular, we can take $\varphi^{\varepsilon}= v^\varepsilon \mathds{1}_{[0,t]}$, and we remark that
\[
\int_0^t \int_{\R^2} \rho^{\varepsilon} v^{\varepsilon} \cdot ([v^{\varepsilon}-\ell^{\varepsilon}] \cdot \nabla v^{\varepsilon}) \, dx ds =
\frac12\int_0^t \int_{\R^2} \rho^{\varepsilon} (v^{\varepsilon}-\ell^{\varepsilon}) \cdot \nabla |v^{\varepsilon}|^2 \, dx ds =0
\]
because $\div v^\varepsilon=0$ and that $[v^{\varepsilon}-\ell^{\varepsilon}]\cdot n \vert_{\partial \mathcal{S}_{0}^\varepsilon} =0$.
Next, we observe that $\partial_{t} v^\varepsilon \in L^2(0,T; \mathcal{V}_R'(\mathcal{F}^{\varepsilon}_{0}))$ and $v^\varepsilon \in L^2(0,T; \mathcal{V}_R(\mathcal{F}^{\varepsilon}_{0}))$ implies that $v^\varepsilon$ is equal for a.e. time to a function $C([0,T]; \mathcal{L}_\varepsilon^2)$ and that
\[
\frac12 \int_{\R^2} \rho^{\varepsilon}(x) \left|v^{\varepsilon}(t,x)\right|^2 \, dx - \frac12 \int_{\R^2} \rho^{\varepsilon}(x) \left|v^{\varepsilon}_{0}(x)\right|^2 \, dx = \int_{0}^t \int_{\R^2} \rho^{\varepsilon}(x) v^{\varepsilon}(t,x) \cdot \partial_{t}v^{\varepsilon}(t,x) \, dx \quad \text{for a.e. } t\in (0,T).
\]
This means that $v^\varepsilon$ satisfies the energy equality 
\begin{equation*}
 \frac12 \int_{\R^2} \rho^{\varepsilon}(x) \left|v^{\varepsilon}(t,x)\right|^2 \, dx + 2\nu \int_0^t \int_{\R^2} |D(v^{\varepsilon})|^2 \, dx ds= 
\frac12 \int_{\R^2} \rho^{\varepsilon}(x) \left|v^{\varepsilon}_0(x)\right|^2 \, dx \quad \text{for a.e. } t\in (0,T),
\end{equation*}
hence, we have for any $\varepsilon$ 
\begin{equation}\label{fsr10}
 \frac12 \int_{\R^2} \rho^{\varepsilon}(t,x) \left|u^{\varepsilon}(t,x)\right|^2 \, dx + 2\nu \int_0^t \int_{\R^2} |D(u^{\varepsilon})|^2 \, dx ds= 
\frac12 \int_{\R^2} \rho^{\varepsilon}_0(x) \left|u^{\varepsilon}_0(x)\right|^2 \, dx \quad\text{for a.e. } t\in (0,T).
\end{equation}

This energy inequality and the hypotheses \eqref{massless} and \eqref{fsr28} imply that
\begin{equation}\label{est:energy}
(u^{\varepsilon})_{\varepsilon} \quad \text{is bounded in} \quad L^\infty(0,T;L^2(\mathbb{R}^2))\cap L^2(0,T;H^1(\mathbb{R}^2)).
\end{equation}

Let us remark that the energy estimate only implies
\[
\varepsilon (h^\varepsilon)'(t), \ \varepsilon^2 (\theta^\varepsilon)'(t) \text{ bounded in } L^\infty(0,T).
\]
These estimates do not allow to localize the rigid body and then to use the method developed in Section~\ref{sec:3}.
The goal of the sequel of this section is to obtain an estimate of $(h^\varepsilon)'$ independently of $\varepsilon$.

\subsection{Semigroup estimates}

The key for the uniform estimate of $\ell^\varepsilon$ is the following theorem concerning the Stokes-rigid body semigroup.
\begin{Theorem} \label{theo Stokes} 
	For each $q\in (1,\infty)$, the semigroup $S^\varepsilon(t)$ on $\mathcal{L}^q_{\varepsilon}$ satisfies the following decay estimates:

	$\bullet$ For $p \in [q, \infty]$, there exists $K_1 = K_1(p,q)>0$ such that for every $v_0^\varepsilon \in \mathcal{L}^{q}_{\varepsilon}$:
		\begin{equation}\label{Lp-Lq}
			\|S^\varepsilon(t)v_0^\varepsilon\|_{\mathcal{L}^p_{\varepsilon}} \leq K_1 t^{\frac{1}{p} - \frac{1}{q}}\|v_0^\varepsilon\|_{\mathcal{L}^q_{\varepsilon}}
			\qquad \text{for all}\quad t>0.
		\end{equation} 

	$\bullet$	For $2\leq q \leq p < \infty,$ there exists $K_2 = K_2(p,q)>0$ such that for every $F^\varepsilon \in L^q( \mathbb{R}^2 ;M_{2\times2}(\mathbb R))$ satisfying $F^\varepsilon=0$ in $\mathcal{S}_0^\varepsilon$:
	\begin{equation}\label{est div 1}
		 \| S^\varepsilon(t) \mathbb P^\varepsilon \div \, F^\varepsilon \|_{\mathcal{L}^p_{\varepsilon}} \leq K_2 t^{-\frac{1}{2} + \frac{1}{p} - \frac{1}{q}} \|F^\varepsilon\|_{L^q( \R^2)} 
		\qquad \text{for all}\quad t>0.
	\end{equation}

	$\bullet$	For $2\leq q < \infty,$ there exists $K_\ell = K_\ell(q)>0$ such that for every $F^\varepsilon \in L^q( \mathbb{R}^2 ;M_{2\times2}(\mathbb R))$ satisfying $F^\varepsilon=0$ on $\mathcal{S}_0^\varepsilon$:
	\begin{equation} \label{eq_lv}
		|\ell_{S^\varepsilon(t) \mathbb P^\varepsilon \div \, F^\varepsilon}| \leq {K_{\ell}} t^{-\left(\frac 12 + \frac 1q \right)} \|F^\varepsilon\|_{L^q(\mathbb R^2)} \qquad \text{for all}\quad t >0.
	\end{equation}
\end{Theorem}
For $\varepsilon$ fixed, estimates like \eqref{Lp-Lq}-\eqref{est div 1} were only established for the Stokes system with the Dirichlet boundary condition \cite{DanShibata, MaremontiSolonnikov}. For the fluid solid problem with one rigid disk in $\R^2$, this result was recently obtained by Ervedoza, Hillairet and Lacave in \cite{EHL}. The only point to check here is that the constants $K_{1}, K_{2}, K_{\ell}$ are independent of $\varepsilon$, which will be easily obtained by a scaling argument. Indeed, as the above estimates are optimal i.e. correspond to the decay of the heat solution, they are invariant to the parabolic scaling of the Navier-Stokes equations.

\begin{proof}
For $\varepsilon=1$, the statements of the theorem were proved in \cite{EHL}, see therein Theorem 1.1, Corollaries 3.10 and 3.11.

We note that $v^\varepsilon(t):=S^\varepsilon(t)v_0^\varepsilon$ satisfies
{\allowdisplaybreaks
\begin{align*}
\frac{\partial v^{\varepsilon}}{\partial t}-\div \sigma(v^{\varepsilon},q^{\varepsilon})
=0 &\quad t>0, \ x\in \mathcal{F}^{\varepsilon}_0, 
\\
\div v^{\varepsilon} = 0 &\quad t>0, \ x\in \mathcal{F}^{\varepsilon}_0,\\
\lim_{|x|\to \infty} v^{\varepsilon}(x)=0 & \quad t>0,\\
v^{\varepsilon}(t,x)=\ell^{\varepsilon}(t)+ r^{\varepsilon}(t) x^\perp & \quad t>0, \ x\in \partial \mathcal{S}_0^{\varepsilon},\\
m^{\varepsilon} (\ell^\varepsilon)'(t)=-\int_{\partial \mathcal{S}_0^{\varepsilon}} \sigma(v^{\varepsilon},q^{\varepsilon})n~d\gamma
& \quad t>0,\\
 J^{\varepsilon} (r^{\varepsilon})'(t)=-\int_{\partial \mathcal{S}_0^{\varepsilon}(t)} x^\perp \cdot \sigma(v^{\varepsilon},q^{\varepsilon})n~d\gamma & \quad t>0,\\
 v^{\varepsilon}(0,\cdot)=v_0^\varepsilon & \quad \text{in } \mathcal{F}^{\varepsilon}_0,\\
 \ell^\varepsilon(0)=\ell_{0}^\varepsilon,\quad
 r^\varepsilon(0)=r_{0}^\varepsilon.
\end{align*}
}
Setting
\begin{equation}\label{mle2.7}
v(t,x):=v^\varepsilon(\varepsilon^2 t, \varepsilon x), \quad 
q(t,x):=\varepsilon q^\varepsilon(\varepsilon^2 t, \varepsilon x), \quad 
\ell(t):=\ell^\varepsilon(\varepsilon^2 t),\quad 
r(t):=\varepsilon r^\varepsilon(\varepsilon^2 t),
\end{equation}
standard calculation gives that
{\allowdisplaybreaks
\begin{align*}
\frac{\partial v}{\partial t}-\div \sigma(v,q)
=0 &\quad t>0, \ x\in \mathcal{F}^{1}_0, 
\\
\div v = 0 &\quad t>0, \ x\in \mathcal{F}^{1}_0,\\
\lim_{|x|\to \infty} v(x)=0 & \quad t>0,\\
v(t,x)=\ell(t)+ r(t) x^\perp & \quad t>0, \ x\in \partial \mathcal{S}_0^{1},\\
m^{1} \ell'(t)=-\int_{\partial \mathcal{S}_0^{1}} \sigma(v,q)n~d\gamma
& \quad t>0,\\
 J^{1} r'(t)=-\int_{\partial \mathcal{S}_0^{1}(t)} x^\perp \cdot \sigma(v,q)n~d\gamma & \quad t>0,\\
 v(0,\cdot)=v_0 & \quad \text{in } \mathcal{F}^{1}_0,\\
 \ell(0)=\ell_{0},\quad
 r(0)=r_{0},
\end{align*}
}
where
\begin{equation}\label{mle2.8}
v_0(x):=v^\varepsilon_0(\varepsilon x), \quad \ell_{0}:=\ell_{0}^\varepsilon, \quad r_{0}:=\varepsilon r_{0}^\varepsilon.
\end{equation}
This means that $v(t)=v^1(t)=S^1(t)v_0$ and thus that
$$
\|v(t)\|_{\mathcal{L}^p_{1}} \leq K_1 t^{\frac{1}{p} - \frac{1}{q}}\|v_0\|_{\mathcal{L}^q_{1}}
			\qquad \text{for all}\quad t>0.
$$
Using \eqref{mle2.7}-\eqref{mle2.8}, this estimate is equivalent to 
$$
\|v^\varepsilon(t)\|_{\mathcal{L}^p_{\varepsilon}} \leq K_1 t^{\frac{1}{p} - \frac{1}{q}} 
\|v_0^\varepsilon\|_{\mathcal{L}^q_{\varepsilon}}
			\qquad \text{for all}\quad t>0.
$$

Relations \eqref{est div 1} and \eqref{eq_lv} can be done similarly. In that case, we also set
$$
F(x):=\frac{1}{\varepsilon} F^\varepsilon( \varepsilon x)
$$
and we show that if $v^\varepsilon(t)=S^\varepsilon(t) \mathbb P^\varepsilon \div F^\varepsilon$, then
$v$ defined by \eqref{mle2.7} satisfies 
$$
v(t)=S^1(t) \mathbb P^1 \div F.
$$

\end{proof}

\subsection{Uniform estimate on the solid velocity}\label{estim_ell}

We first show that there exists $\lambda_0>0$ such that if 
$\| v^\varepsilon_{0} \|_{\mathcal{L}^2_{\varepsilon}} \leq \lambda_{0}$, then there exists a unique 
\begin{equation}\label{mle0.2}
v^\varepsilon\in \mathcal{C}^0([0,T];\mathcal{L}^2_\varepsilon)\cap \mathcal{C}^0_{3/8}([0,T];\mathcal{L}^8_\varepsilon)
\quad \text{with}\quad \ell^\varepsilon:=\ell_{v^\varepsilon}\in \mathcal{C}^0_{1/2}([0,T];\mathbb{R}^2)
\end{equation}
satisfying \eqref{Duhamel} (that is a {\it mild} solution).
Here we have denoted for any Banach space $X$ by $\mathcal{C}^0_{\alpha}([0,T];X)$ the Banach space of functions $f$ such that
$t\mapsto t^\alpha f(t)$ are continuous from $[0,T]$ in $X$. The norm associated is
$$
\|f\|_{\mathcal{C}^0_{\alpha}([0,T];X)}:=\sup_{t\in [0,T]} t^\alpha \|f(t)\|_X.
$$

\begin{Proposition}\label{P01}
There exist $\lambda_{0},\mu_{0}>0$ independent of $\varepsilon$ such that the following holds for any $T>0$. If $v^\varepsilon_{0}\in \mathcal{L}^2_{\varepsilon}$ satisfies
\begin{equation}\label{Smallness-v0}
\| v^\varepsilon_{0} \|_{\mathcal{L}^2_{\varepsilon}} \leq \lambda_{0} 
\end{equation}
then there exists a unique $v^\varepsilon$ satisfying \eqref{Duhamel} and such that 
\begin{equation*}
 \| v^\varepsilon \|_{\mathcal{C}^0([0,T];\mathcal{L}^2_{\varepsilon})} , \ \| v^\varepsilon \|_{\mathcal{C}^0_{3/8}([0,T];\mathcal{L}^8_{\varepsilon})}, \ \|\ell^\varepsilon \|_{\mathcal{C}^0_{1/2}([0,T];\R^2)} \text{ are bounded by } \mu_{0}.
\end{equation*}
Moreover there exists a constant $C>0$ independent of $\varepsilon$ such that, if $v^\varepsilon_{0a}, v^\varepsilon_{0b}\in \mathcal{L}^2_{\varepsilon}$ are two initial conditions satisfying \eqref{Smallness-v0}, then 
\begin{equation}\label{mle0.3}
\| v^\varepsilon_{a} - v^\varepsilon_{b}\|_{\mathcal{C}^0([0,T];\mathcal{L}^2_\varepsilon)} 
 \leq C \| v^\varepsilon_{0a} - v^\varepsilon_{0b}\|_{\mathcal{L}^2_\varepsilon}. 
\end{equation}
 \end{Proposition}

\begin{proof}
As we have proved in the previous theorem that the constant in the semigroup estimates are independent of $\varepsilon$, it is then enough to follow the fixed point argument in \cite[pp. 364-371]{EHL}. For completeness, let us write here the details.

\medskip

Let us introduce the space
\begin{equation*}
\mathcal{X}^\varepsilon:=\left\{ v^\varepsilon \in \mathcal{C}^0([0,T];\mathcal{L}^2_\varepsilon)\cap \mathcal{C}^0_{3/8}([0,T];\mathcal{L}^8_\varepsilon)
\quad \text{with}\quad \ell_{v^\varepsilon}\in \mathcal{C}^0_{1/2}([0,T];\mathbb{R}^2) \right\}
\end{equation*}
endowed with the norm
$$
\| v^\varepsilon\|_{\mathcal{X}^\varepsilon} := \|v^\varepsilon\|_{\mathcal{C}^0([0,T];\mathcal{L}^2_\varepsilon)}+
\|v^\varepsilon\|_{\mathcal{C}^0_{3/8}([0,T];\mathcal{L}^8_\varepsilon)}
+\|\ell_{v^\varepsilon}\|_{\mathcal{C}^0_{1/2}([0,T];\mathbb{R}^2)}.
$$
Let us also define the map
\begin{equation*}
\mathcal{Z}^\varepsilon : \mathcal{X}^\varepsilon \to \mathcal{X}^\varepsilon, 
\end{equation*}
defined by
$$
\mathcal{Z}^\varepsilon(v^\varepsilon)(t)= 
 S^\varepsilon(t) v_{0}^\varepsilon + \displaystyle \int_{0}^t S^\varepsilon(t-s) \mathbb{P}^\varepsilon \div F^\varepsilon(v^\varepsilon)(s)\, ds,
$$
where $F^\varepsilon$ is defined by \eqref{eq:Feps}. One can define 
$$
\Phi(v^\varepsilon,w^\varepsilon)(t)=\displaystyle \int_{0}^t S^\varepsilon(t-s) \mathbb{P}^\varepsilon \div G^\varepsilon(v^\varepsilon,w^\varepsilon)(s)\, ds,
$$
where
\begin{equation*}
	G^\varepsilon (v^\varepsilon,w^\varepsilon)= \left\{
	\begin{array}{rcll}
		v^\varepsilon\otimes (\ell_{w^\varepsilon}- w^\varepsilon) & \text{ on $\mathcal{F}_{0}^\varepsilon$}\,\\
		 0 &\text{ on $\mathcal{S}_{0}^\varepsilon$}.
	 \end{array}\right.
\end{equation*}

We deduce from \eqref{est div 1} that
$$
t^{\frac38} \|\Phi(v^\varepsilon,w^\varepsilon)(t)\|_{\mathcal{L}^8_\varepsilon} \leq t^{\frac38} K_2(8,4) \int_0^t (t-s)^{-5/8} 
\left( \|v^\varepsilon(s)\otimes w^\varepsilon(s)\|_{\mathcal{L}^4_\varepsilon}
+ |\ell_{w^\varepsilon}(s)| \|v^\varepsilon(s)\|_{\mathcal{L}^4_\varepsilon}
\right) \, ds.
$$
Using H\"older's inequalities, we obtain from the above inequality that
\begin{align}
t^{\frac38} \|\Phi(v^\varepsilon,w^\varepsilon)(t)\|_{\mathcal{L}^8_\varepsilon} 
&\leq t^{\frac38} K_2(8,4) \int_0^t (t-s)^{-5/8} 
\Big( s^{-\frac38} \|v^\varepsilon\|_{\mathcal{C}^0_{3/8}\mathcal{L}^8_\varepsilon} s^{-\frac38} \|w^\varepsilon\|_{\mathcal{C}^0_{3/8}\mathcal{L}^8_\varepsilon} \nonumber\\
&\hspace{4.5cm}+ s^{-\frac12}|\ell_{w^\varepsilon}|_{\mathcal{C}^0_{1/2}} \|v^\varepsilon\|_{L^\infty\mathcal{L}^2_\varepsilon}^{1/3} (s^{-\frac38})^{2/3} \|v^\varepsilon\|_{\mathcal{C}^0_{3/8}\mathcal{L}^8_\varepsilon}^{2/3}
\Big) \, ds \nonumber
\\
&\leq 2K_2(8,4) B\left(\frac 58,\frac 34\right) \|v^\varepsilon\|_{\mathcal{X}^\varepsilon} \|w^\varepsilon\|_{\mathcal{X}^\varepsilon},\label{mle1.8}
\end{align}
where $B(\cdot,\cdot)$ is the Beta function:
\[
B(\alpha,\beta):=\int_{0}^1 (1-\tau)^{-\alpha} \tau^{-\beta}\, d\tau.
\]

Similarly,
\begin{equation}\label{mle1.6}
\|\Phi(v^\varepsilon,w^\varepsilon)(t)\|_{\mathcal{L}^2_\varepsilon} 
\leq 2K_2(2,2) B\left(\frac 12,\frac 12\right) \|v^\varepsilon\|_{\mathcal{X}^\varepsilon} \|w^\varepsilon\|_{\mathcal{X}^\varepsilon}.
\end{equation}
Finally,
$$
\ell_{\Phi(v^\varepsilon,w^\varepsilon)(t)}=\displaystyle \int_{0}^t \ell_{S^\varepsilon(t-s) \mathbb{P}^\varepsilon \div G^\varepsilon(v^\varepsilon,w^\varepsilon)(s)}\, ds,
$$
and from \eqref{eq_lv}, we deduce
$$
t^{\frac12} |\ell_{\Phi(v^\varepsilon,w^\varepsilon)(t)}|
\leq \displaystyle t^{\frac12} \int_{0}^t K_\ell(4) (t-s)^{-\left(\frac{1}{2}+\frac{1}{4}\right)} \|G^\varepsilon(v^\varepsilon,w^\varepsilon)(s)\|_{L^4(\R^2)} \, ds.
$$
With the same estimates as in \eqref{mle1.8}, we obtain
\begin{equation}\label{mle1.7}
t^{\frac12} |\ell_{\Phi(v^\varepsilon,w^\varepsilon)(t)}|
\leq 2K_\ell(4) B\left(\frac 34,\frac 34\right) \|v^\varepsilon\|_{\mathcal{X}^\varepsilon} \|w^\varepsilon\|_{\mathcal{X}^\varepsilon}.
\end{equation}
Gathering \eqref{mle1.8}, \eqref{mle1.6} and \eqref{mle1.7} yields 
\begin{equation}\label{mle1.9}
\|\Phi(v^\varepsilon,w^\varepsilon)\|_{\mathcal{X}^\varepsilon} \leq C_0 \|v^\varepsilon\|_{\mathcal{X}^\varepsilon}\|w^\varepsilon\|_{\mathcal{X}^\varepsilon},
\end{equation}
where
\begin{equation*}
C_0=2 \left(K_2(8,4) B\left(\frac 58,\frac 34\right) +K_2(2,2) B\left(\frac 12,\frac 12\right) + K_\ell(4) B\left(\frac 34,\frac 34\right) \right).
\end{equation*}

We assume \eqref{Smallness-v0} for $\lambda_0$ that we fix below and we apply \eqref{Lp-Lq} in order to obtain 
\begin{equation}\label{mle2.1}
\| S^\varepsilon v_{0}^\varepsilon\|_{\mathcal{X}^\varepsilon} \leq C_1\lambda_0,
\end{equation}
where 
\begin{equation*}
C_1= K_1(8,2) + K_1(2,2) + K_1(\infty,2).
\end{equation*}
Note that we have used in \eqref{mle2.1} the relation
$$
 |\ell_{S^\varepsilon(t) v_{0}^\varepsilon}|
\leq \|S^\varepsilon(t)v_0^\varepsilon\|_{\mathcal{L}^\infty_{\varepsilon}} 
$$
which is a consequence of \eqref{ellv}.

Relations \eqref{mle1.9} and \eqref{mle2.1} imply that the mapping $\mathcal{Z}^\varepsilon$ is well-defined and that
\begin{equation*}
\| \mathcal{Z}^\varepsilon(v^\varepsilon)\|_{\mathcal{X}^\varepsilon} \leq C_1\lambda_0 +C_0 \|v^\varepsilon\|_{\mathcal{X}^\varepsilon}^2.
\end{equation*}
Let us set
\begin{equation}\label{def:R}
 R:=\frac{1}{4C_0}\quad \text{and}\quad \lambda_0 = \min\Big( \frac{R}{2C_{1}},R\Big).
\end{equation}
Then the closed ball $B_{\mathcal{X}^\varepsilon}(0,R)$ of $\mathcal{X}^\varepsilon$ is invariant by $\mathcal{Z}^\varepsilon$ and 
if $v^\varepsilon, w^\varepsilon\in B_{\mathcal{X}^\varepsilon}(0,R)$, 
\begin{equation*}
\| \mathcal{Z}^\varepsilon(v^\varepsilon)-\mathcal{Z}^\varepsilon(w^\varepsilon)\|_{\mathcal{X}^\varepsilon} 
=\| \Phi(v^\varepsilon,v^\varepsilon-w^\varepsilon) + \Phi(v^\varepsilon-w^\varepsilon, w^\varepsilon)\|_{\mathcal{X}^\varepsilon} 
\leq 2 C_0 R \|v^\varepsilon-w^\varepsilon\|_{\mathcal{X}^\varepsilon}=\frac12 \|v^\varepsilon-w^\varepsilon\|_{\mathcal{X}^\varepsilon}.
\end{equation*}

Using the Banach fixed point, we deduce the existence and uniqueness results. Moreover, we have $\mu_{0}=R=1/(4C_{0})$, which is independent of $\varepsilon$. 
This strategy comes from \cite{Kato} and \cite{KatoFujita62}. It is originally done through an iterative method, but it was adapted as a fixed point argument in \cite{Mon2006a} (see also \cite{SilTak2012a}).

{\it Sensitivity of $v^\varepsilon$ to the initial data.} 
Assume $v^\varepsilon_{0a}, v^\varepsilon_{0b}\in \mathcal{L}^2_{\varepsilon}$ satisfy \eqref{Smallness-v0} with $\lambda_{0}$ defined in \eqref{def:R}.
Then, 
\begin{align*}
\| v^\varepsilon_{a} - v^\varepsilon_{b}\|_{\mathcal{X}^\varepsilon} 
& \leq \| S^\varepsilon (v_{0a}^\varepsilon-v_{0b}^\varepsilon)\|_{\mathcal{X}^\varepsilon} 
 +\left\| \Phi(v^\varepsilon_{a},v^\varepsilon_{a} - v^\varepsilon_{b})\right\|_{\mathcal{X}^\varepsilon} 
 +\left\| \Phi(v^\varepsilon_{a} - v^\varepsilon_{b}, v^\varepsilon_{b})\right\|_{\mathcal{X}^\varepsilon} 
 \\
& \leq C_1 \| v^\varepsilon_{0a} - v^\varepsilon_{0b}\|_{\mathcal{L}^2_\varepsilon} 
 +C_0 \| v^\varepsilon_{a} - v^\varepsilon_{b}\|_{\mathcal{X}^\varepsilon} 
 \left( \| v^\varepsilon_{a} \|_{\mathcal{X}^\varepsilon} 
 +\| v^\varepsilon_{b}\|_{\mathcal{X}^\varepsilon} \right)\\
 & \leq C_1 \| v^\varepsilon_{0a} - v^\varepsilon_{0b}\|_{\mathcal{L}^2_\varepsilon} 
 +2C_{0}\mu_{0} \| v^\varepsilon_{a} - v^\varepsilon_{b}\|_{\mathcal{X}^\varepsilon} .
\end{align*}
We conclude that
$$
\| v^\varepsilon_{a} - v^\varepsilon_{b}\|_{\mathcal{X}^\varepsilon} \leq 2 C_{1}\| v^\varepsilon_{0a} - v^\varepsilon_{0b}\|_{\mathcal{L}^2_\varepsilon}.
$$
\end{proof}

One can show that a mild solution in the above sense is also a weak solution (see \cite{EHL}). For sake of completeness, we give the proof of this result here.

\begin{Lemma}
Assume $v^\varepsilon$ satisfies \eqref{mle0.2}-\eqref{Smallness-v0} and \eqref{Duhamel}. Then $v^\varepsilon$ is the weak solution of \eqref{mle0.4}-\eqref{mle0.5} 
in the sense of Definition~\ref{D02}.
\end{Lemma}
\begin{proof}
Let us consider a sequence $(v_{0n}^\varepsilon)_n$ with values in $\mathcal{D}((A^\varepsilon)^{1/2})$ such that
$$
v_{0n}^\varepsilon \to v_0^\varepsilon \quad \text{in} \ \mathcal{L}^2_{\varepsilon}
$$ 
and such that $v_{0n}^\varepsilon$ satisfies \eqref{Smallness-v0}.
For all $n$, it is proved in \cite{Takahashi&Tucsnak04} that there exists a unique strong solution
$$
v_n^\varepsilon \in H^1(0,T;\mathcal{L}^2_\varepsilon)\cap C([0,T];\mathcal{D}((A^\varepsilon)^{1/2})) \cap L^2(0,T;\mathcal{D}(A^\varepsilon))
$$
of \eqref{mle0.4}-\eqref{mle0.5} and it satisfies \eqref{Duhamel} since in this case
$$
\div F^\varepsilon(v^\varepsilon_n) \in L^2(0,T;L^2(\mathbb{R}^2)). 
$$
It is also proved in \cite{Takahashi&Tucsnak04} that $(v_n^\varepsilon)$ converges towards the weak solution of \eqref{mle0.4}-\eqref{mle0.5} 
in the sense of Definition~\ref{D02}.

From \eqref{mle0.3}, we also have that $(v_n^\varepsilon)$ converges towards the mild solution of Proposition~\ref{P01} in $L^\infty(0,T;\mathcal{L}^2_\varepsilon)$.

Consequently, the mild solution $v^\varepsilon$, that satisfies \eqref{mle0.2}-\eqref{Smallness-v0}, is the weak solution associated to $v_0^\varepsilon$.

\end{proof}

\section{Proof of Theorem~\ref{mainthm2}}\label{sec:3}

This section is dedicated to the proof of Theorem~\ref{mainthm2}. As recalled in the introduction, \cite{Takahashi&Tucsnak04} established that for any $\varepsilon>0$, there exists a unique weak solution $(u^{\varepsilon}, h^{\varepsilon}, \theta^{\varepsilon})$ to problem \eqref{eq01}--\eqref{fsr07} in the sense of Definition~\ref{D01}. Let us fix $T>0$.

\subsection{First convergences}\label{sec:conv}

Thanks to the energy estimate \eqref{est:energy}, we can extract a subsequence such that
\begin{equation}\label{rev00}
u^{\varepsilon} \overset{*}{\rightharpoonup} u \quad \quad \text{in} \quad L^\infty(0,T;L^2(\mathbb{R}^2))\cap L^2(0,T;H^1(\mathbb{R}^2)).
\end{equation}
By abuse of notation, we continue to write $u^{\varepsilon}$ the subsequence. 

If the initial data satisfies the smallness condition \eqref{petitesse} with $\lambda_{0}$ given in Proposition~\ref{P01}, we deduce from Section~\ref{estim_ell} that
\begin{equation*}
|(h^\varepsilon)'(t)| \leq \frac{\mu_{0}}{\sqrt{t}} \quad (t>0),
\end{equation*}
where $\mu_{0}$ is independent of $\varepsilon$. As a consequence,
\begin{equation*}
|h^\varepsilon(t)| \leq 2\mu_{0}\sqrt{T} \quad (t\in [0,T]).
\end{equation*}
We fix $q\in (1,2)$, thus $(h^\varepsilon)$ is bounded in $W^{1,q}(0,T;\mathbb{R}^2)$ and
we have, up to a subsequence,
\begin{equation}\label{rev:convh}
h^{\varepsilon} \to h \quad \text{uniformly in} \ [0,T],
\end{equation}
with
\begin{equation}\label{rev02}
h\in W^{1,q}(0,T).
\end{equation}

\subsection{Modified test functions}\label{sec:cutoff}

The key to treat shrinking obstacles problem is to approximate test functions in $\mathbb{R}^2$ by admissible test functions in the perforated domain.

\begin{Proposition}\label{lem:bogo}
Let $T>0$, $\varphi \in C^\infty_c([0,T)\times \mathbb{R}^2)$ with $\div \varphi=0$.
We consider $q\in (1,2)$ as in Section~\ref{sec:conv}.
For any $\eta>0$, there exists $\varphi^\eta \in W^{1,q}_c([0,T);H^1(\mathbb{R}^2))$ satisfying
\begin{equation}\label{fsr4.1}
\div \varphi^\eta =0 \quad \text{in} \ [0,T)\times \mathbb{R}^2,
\end{equation}
\begin{equation}\label{fsr4.8}
\varphi^\eta \equiv 0 \quad t\in [0,T), \quad x\in B\left(h(t),\frac{\eta}{2}\right),
\end{equation}
\begin{equation}\label{fsr4.4}
 \varphi^\eta \overset{*}{\rightharpoonup} \varphi \quad L^\infty(0,T;H^1(\mathbb{R}^2)),
\end{equation}
\begin{equation}\label{fsr4.6}
 \partial_t \varphi^\eta \rightharpoonup \partial_t \varphi \quad L^q(0,T;L^2(\mathbb{R}^2)).
\end{equation}
\end{Proposition}

\begin{proof}
We introduce a cutoff function $\chi \in C^{\infty}(\mathbb{R}^2,[0,1])$ such that $\chi \equiv 1$ in $B(0,1)^c$ and $\chi \equiv 0$ in $B(0,1/2)$. Let us denote the annulus $B(0,1) \setminus \overline{B(0,1/2)}$ by $A$.

We remark that the function $\tilde \varphi^\eta: (t,y)\mapsto \varphi(t,\eta y + h(t)) \nabla \chi(y)$ belongs to $W^{1,q}(0,T;L^2(A))$ and verifies for any $t$
\begin{align*}
 \int_{A} \tilde \varphi^\eta (t,y)\, dy & = \int_{A} \div\Big( \varphi(t,\eta y + h(t)) \chi(y) \Big) \, dy = \int_{\partial B(0,1)} \varphi(t,\eta y + h(t)) \cdot n(y) \, ds\\
 & = \int_{B(0,1)} \div\Big( \varphi(t,\eta y + h(t)) \Big) \, dy =0,
\end{align*}
where we have used twice that $\varphi$ is divergence free, that $\chi \equiv 1$ on $\partial B(0,1)$ and that $\chi \equiv 0$ on $\partial B(0,1/2)$.
With these properties, it is known by \cite[Theorem III.3.1]{GaldiBook} (and Exercice III.3.6) that there exists $C$ depending only on $A$ such that the problem
\begin{gather*}
 \div \tilde g^\eta = \tilde \varphi^\eta, \quad \tilde g^\eta \in W^{1,q}(0,T;H^1_{0}(A))
\end{gather*}
has a solution such that
\begin{gather*}
 \| \tilde g^\eta \|_{L^\infty(0,T;H^1(A))} \leq C \| \tilde \varphi^\eta \|_{L^\infty(0,T;L^2 (A))}, \\
 \| \partial_{t} \tilde g^\eta \|_{L^q(0,T;H^1(A))} \leq C \| \partial_{t} \tilde \varphi^\eta \|_{L^q(0,T;L^2 (A))} .
\end{gather*}
Extending $\tilde g^\eta$ by zero in the exterior of $A$, we define
\[
\varphi^{\eta}(t,x)=\varphi(t,x) \chi \Big( \frac{x-h(t)}\eta\Big) - g^\eta(t,x)
\]
where
\[
g^\eta(t,x):= \tilde g^\eta \Big( t, \frac{x-h(t)}{\eta} \Big).
\]
We easily verify the divergence free condition \eqref{fsr4.1}. Moreover, with a change of variables, we also note that
\begin{equation}\label{g-nablag}
 \frac1\eta \Big\| g^\eta \Big\|_{L^\infty(0,T;L^2(\mathbb{R}^2))} + \Big\| \nabla g^\eta \Big\|_{L^\infty(0,T;L^2(\mathbb{R}^2))} 
 \leq C\|\tilde \varphi^\eta\|_{L^\infty(0,T;L^2 (A))}
\leq C \| \varphi \|_{L^\infty((0,T)\times \mathbb{R}^2)} 
\end{equation}
so we check that
\[
\frac1\eta \| \varphi^\eta -\varphi \|_{L^\infty(0,T;L^2(\mathbb{R}^2))} + \| \nabla \varphi^\eta-\nabla \varphi \|_{L^\infty(0,T;L^2(\mathbb{R}^2))} \leq C \| \varphi \|_{W^{1,\infty}((0,T)\times \mathbb{R}^2)} 
\]
which gives directly that $\varphi^\eta$ converges to $\varphi$ strongly in $L^\infty(0,T; L^2(\mathbb{R}^2))$ and weak-$*$ in $L^\infty(0,T; H^1(\mathbb{R}^2))$. By uniqueness of the limit, we do not need to extract a subsequence and we get the weak limit \eqref{fsr4.4}.

Now we compute
\begin{align*}
 \partial_{t} \varphi^\eta(t,x) -\partial_{t} \varphi(t,x) = &\partial_{t} \varphi(t,x) \Big( \chi \Big( \frac{x-h(t)}\eta\Big) -1\Big)- \frac{\varphi(t,x)}\eta h'(t) \cdot (\nabla \chi)\Big( \frac{x-h(t)}\eta\Big)\\
 &- \partial_{t} \tilde g^\eta \Big( t, \frac{x-h(t)}{\eta} \Big) 
+ h'(t) \cdot \nabla g^\eta ( t, x ) .
\end{align*}
It is obvious that the first right hand side term converges to zero strongly in $L^\infty(0,T;L^2(\mathbb{R}^2))$. It is also an easy computation to check that the second term is bounded in $L^q(0,T;L^2(\mathbb{R}^2))$ and tends to zero strongly in $L^q(0,T;L^p(\mathbb{R}^2))$ for $p\in [1,2)$. Hence, it converges weakly to zero in 
$L^q(0,T;L^2(\mathbb{R}^2))$. From \eqref{rev02} and \eqref{g-nablag}, we know that the last term is bounded in $L^q(0,T;L^2(\mathbb{R}^2))$, and as it converges to zero in $\mathcal{D}' ((0,T)\times \R^2)$, we infer that it converges also weakly to zero in $L^q(0,T;L^2(\mathbb{R}^2))$. Finally, we note that
\[
 \frac1\eta \Big\| \partial_t \tilde g^\eta \Big( t, \frac{x-h(t)}{\eta}\Big) \Big\|_{L^\infty(0,T;L^2(\mathbb{R}^2))} 
\leq C \| \partial_{t} \varphi \|_{L^\infty((0,T)\times \mathbb{R}^2)} ,
\]
hence the third right hand side term tends to zero strongly in $L^\infty(0,T;L^2(\mathbb{R}^2))$.
 This gives \eqref{fsr4.6}.

Due to the support of $\chi$ and $g^\eta$, it is clear that $\varphi^\eta \equiv 0$ on $B(h(t),\frac{\eta}{2})$. This ends the proof.
\end{proof}

\begin{Remark}\label{rm:varphi}
An important consequence is the approximation of any test function. Let $T>0$ and $\varphi \in C^\infty_c([0,T)\times \R^2)$ with $\div \varphi=0$. Then, we have constructed a family $(\varphi^\eta)_{\eta}$ of divergence free test functions which tends to $\varphi$ in the sense of \eqref{fsr4.4}-\eqref{fsr4.6}.
Moreover, for any $\eta>0$ fixed, we put together the strong convergence of $h^\e$ \eqref{rev:convh} with the support of $\varphi^\eta$ \eqref{fsr4.8} to deduce the existence of $\varepsilon_\eta>0$ such that
\begin{equation*}
\varphi^\eta \equiv 0 \quad \text{for }t\in 
[0,T), \quad x\in \mathcal{S}^\varepsilon(t), \quad \varepsilon\leq \varepsilon_{\eta}.
\end{equation*}
This implies that $\varphi^\eta$ is an admissible test function for the fluid-solid problem (see Definition~\ref{D01}).
\end{Remark}
\begin{Remark}
In the proof of the above proposition, we note that $H^1$ is the critical space in dimension two: $\chi(\frac\cdot\eta) -1 $ tends to zero strongly in $W^{1,p}$ for any $p\in [1,2)$, is bounded in $H^1$ (then tends weakly to zero), and goes to infinity in $W^{1,p}$ for $p>2$. This explains why the standard framework for shrinking obstacles problems is $H^1$ (see, e.g., \cite{Allaire90a,Allaire90b,CM82,IftKel2009a,IftLopNus2006a,Tartar80}). Nevertheless, as we need an estimate of the solid velocities, it is natural to look for a $C^0$ estimate of the velocity, hence a $H^s$ estimate for $s>1$. Unfortunately, the $H^2$ analysis developed in \cite{DasRob2011b,SilTak2014b} fails (see \cite{CPRXpreprint}).

In dimension three, the critical space for the cutoff argument is $W^{1,3}$ which is again not embedded in $C^0$.
\end{Remark}

\subsection{Passing to the limit in the Navier-Stokes equations}\label{sec:limit}

The first step is to pass to the limit $\varepsilon\to 0$ for $\eta$ fixed.

\begin{Theorem}\label{thmeta}
 Let $T>0$ and let $\varphi \in C^\infty_c([0,T)\times \R^2)$ with $\div \varphi=0$. 
 We consider the family $(\varphi^\eta)_{\eta>0}$ obtained in Proposition~\ref{lem:bogo}. Then, for any $\eta>0$, the limit $u$ of $u^\e$ (see \eqref{rev00}) verifies
\begin{equation*}
-\int_0^T \int_{\R^2} u \cdot \left(\frac{\partial \varphi^{\eta}}{\partial t} + (u \cdot \nabla)\varphi^{\eta} \right) \, dx 
+2\nu \int_0^T \int_{\R^2} D(u): D(\varphi^{\eta}) \, dx = \int_{\R^2} u^0(x) \cdot \varphi^{\eta}(0,x) \, dx.
\end{equation*}
\end{Theorem}
\begin{proof}
Let $\eta>0$ be fixed. 
From \eqref{rev02} and Sobolev embeddings, we know that $t\mapsto h(t)$ is continuous on $[0,T]$ and then uniformly continuous. Hence, there exists a uniform subdivision $t_0=0<t_1<\ldots<t_{M+1}=T$ such that
for any $t\in (t_j, t_{j+1})$, 
$$
|h(t)-h(t_j)| \leq \frac{\eta}{6}.
$$
From \eqref{fsr4.8} in Proposition~\ref{lem:bogo}, we deduce 
$$
\varphi^{\eta} \equiv 0 \quad \text{in} \ (t_j, t_{j+1})\times B(h(t_j),\frac{\eta}{3}).
$$
Putting together this relation with \eqref{rev:convh}, there exist open relatively compact sets $\mathcal{O}_j$ and $\tilde \varepsilon_{\eta}>0$ such that for all $\varepsilon<\tilde\varepsilon_{\eta}$
\begin{equation}\label{fsr5.0}
\mathcal{S}^\varepsilon(t) \cap \mathcal{O}_j=\emptyset \quad \text{for all} \ t\in (t_j, t_{j+1})
\quad \text{and} \quad
\supp \varphi^{\eta} \subset \sum_{j=0}^M (t_j, t_{j+1}) \times \mathcal{O}_j.
\end{equation}

For any $j=0,\dots ,M$, we write the Helmholtz-Weyl decomposition
\begin{equation*}
u^{\varepsilon} = \mathbb{P}_{\mathcal{O}_j} u^{\varepsilon} + \nabla q^{\varepsilon},
\end{equation*}
where $\mathbb{P}_{\mathcal{O}_j}$ is the Leray projection on $\mathcal{H}(\mathcal{O}_{j})$ (see the introduction for the definition of $\mathcal{H}(\mathcal{O})$).
This projection is orthogonal in $L^2$ and by a standard estimate on the Laplace problem with Neumann boundary condition, there exists a constant $C_{\mathcal{O}_j}>0$ such that
\begin{equation*}
\| \mathbb{P}_{\mathcal{O}_j} u^{\varepsilon} \|_{L^2(\mathcal{O}_j)} \leq \| u^{\varepsilon} \|_{L^2(\mathcal{O}_j)}
\quad\text{and}\quad
\| \mathbb{P}_{\mathcal{O}_j} u^{\varepsilon} \|_{H^1(\mathcal{O}_j)} \leq C_{\mathcal{O}_j} \| u^{\varepsilon} \|_{H^1(\mathcal{O}_j)}.
\end{equation*}
Thus, by \eqref{est:energy},
\begin{equation*}
\left( \mathbb{P}_{\mathcal{O}_j} u^{\varepsilon} , \nabla q^{\varepsilon} \right)_{\varepsilon} \quad \text{is bounded in}\ L^\infty(0,T; L^2(\mathcal{O}_j))\cap L^2(0,T; H^1(\mathcal{O}_j)).
\end{equation*}
In particular,
\begin{gather}
 \mathbb{P}_{\mathcal{O}_j} u^{\varepsilon} \overset{*}{\rightharpoonup} \mathbb{P}_{\mathcal{O}_j} u \quad \text{in}\ L^\infty(0,T; L^2(\mathcal{O}_j))\cap L^2(0,T; H^1(\mathcal{O}_j)),\label{fsr7.3} \\
 \nabla q^{\varepsilon} \overset{*}{\rightharpoonup} \nabla q =u- \mathbb{P}_{\mathcal{O}_j} u \quad \text{in}\ L^\infty(0,T; L^2(\mathcal{O}_j))\cap L^2(0,T; H^1(\mathcal{O}_j)). \label{fsr7.4}
\end{gather}

Now we derive a time estimate for $ \mathbb{P}_{\mathcal{O}_j} u^{\varepsilon}$ in order to get a strong convergence. For any divergence free test function $\psi \in C^{\infty}_c((t_j,t_{j+1})\times \mathcal{O}_j)$, we have by \eqref{fsr5.0} that $\psi(t,\cdot) \in \mathcal{V}_R(\mathcal{F}^{\varepsilon}(t))$ (see \eqref{fsr09}), hence \eqref{fsr08} gives
\begin{align*}
\langle \partial_t \mathbb{P}_{\mathcal{O}_j} u^{\varepsilon},\psi \rangle_{L^2((t_j,t_{j+1});\mathcal{V}(\mathcal{O}_j)'),L^2((t_j,t_{j+1});\mathcal{V}(\mathcal{O}_j))} 
=& - \int_{t_j}^{t_{j+1}} \int_{\mathcal{O}_j} \mathbb{P}_{\mathcal{O}_j} u^{\varepsilon}\cdot \partial_t \psi \, dxdt
\\
=& - \int_{t_j}^{t_{j+1}} \int_{\mathcal{O}_j} u^{\varepsilon} \cdot \partial_t \psi \, dx dt\\
=&
\int_{t_j}^{t_{j+1}} \int_{\mathcal{O}_j} u^{\varepsilon}\cdot (u^{\varepsilon} \cdot \nabla)\psi \, dx dt
-2\nu \int_{t_j}^{t_{j+1}} \int_{\mathcal{O}_j} D(u^{\varepsilon}): D(\psi) \, dx dt.
\end{align*}
Thus, by using \eqref{est:energy} and the interpolation inequality $\|f\|_{L^4(\R^2)} \leq \|f\|_{L^2(\R^2)}^{1/2}\|\nabla f\|_{L^2(\R^2)}^{1/2}$, we get 
\begin{align*}
\Big|\langle \partial_t \mathbb{P}_{\mathcal{O}_j}& u^{\varepsilon},\psi \rangle_{L^2((t_j,t_{j+1});\mathcal{V}(\mathcal{O}_j)'),L^2((t_j,t_{j+1});\mathcal{V}(\mathcal{O}_j))}\Big|
\\
&\leq 
\| u^{\varepsilon} \|_{L^4((t_j,t_{j+1});L^4(\mathcal{O}_j))}^2 \| \psi \|_{L^2((t_j,t_{j+1});\mathcal{V}(\mathcal{O}_j))}
+\| D u^{\varepsilon} \|_{L^2((t_j,t_{j+1});L^2(\mathcal{O}_j))} \| \psi \|_{L^2((t_j,t_{j+1});\mathcal{V}(\mathcal{O}_j))}
\\
&\leq C \| \psi \|_{L^2((t_j,t_{j+1});\mathcal{V}(\mathcal{O}_j))}.
\end{align*}
Consequently, $\left(\partial_{t} \mathbb{P}_{\mathcal{O}_j} u^{\varepsilon} \right)_{\varepsilon}$ is bounded in $L^2((t_j, t_{j+1}); \mathcal{V}(\mathcal{O}_j)' )$, and the Aubin-Lions lemma in $H^{1} \cap \mathcal{H}(\mathcal{O}_{j}) \hookrightarrow L^4\cap \mathcal{H}(\mathcal{O}_{j}) \hookrightarrow \mathcal{V}'(\mathcal{O}_{j})$ allows us to extract a subsequence such that
\begin{equation}\label{fsr7.1}
 \mathbb{P}_{\mathcal{O}_j} u^{\varepsilon} \to \mathbb{P}_{\mathcal{O}_j} u \quad \text{strongly in}\ L^2((t_j, t_{j+1}); L^4(\mathcal{O}_j)).
\end{equation}
Actually, by the uniqueness of the limit, we do not need to extract a subsequence in \eqref{fsr7.1}.

These convergences are enough to pass to the limit in the Navier-Stokes equations. Indeed, for any $\varepsilon\in (0,\tilde\varepsilon_{\eta}]$, we know from \eqref{fsr5.0} that $\varphi^\eta$ is an admissible test function, and \eqref{fsr08} reads
\begin{multline*}
-\int_0^T \int_{\R^2} u^{\varepsilon} \cdot \frac{\partial \varphi^{\eta}}{\partial t} \, dx dt
-\sum_{j=0}^M \int_{(t_j, t_{j+1})\times \mathcal{O}_j} (u^{\varepsilon}\otimes u^{\varepsilon}) : \nabla\varphi^{\eta} \, dx dt
+2\nu \int_0^T \int_{\R^2} D(u^{\varepsilon}): D(\varphi^{\eta}) \, dx dt 
\\
= \int_{\R^2} u^{\varepsilon}_0(x) \cdot \varphi^{\eta}(0,x) \, dx.
\end{multline*}
Using the weak limits \eqref{rev00} and \eqref{fsr28}, we easily pass to the limit in the linear term
\begin{multline*}
-\int_0^T \int_{\R^2} u^{\varepsilon} \cdot \frac{\partial \varphi^{\eta}}{\partial t} \, dx dt
+2\nu \int_0^T \int_{\R^2} D(u^{\varepsilon}): D(\varphi^{\eta}) \, dx dt 
\\
\to
-\int_0^T \int_{\R^2} u \cdot \frac{\partial \varphi^{\eta}}{\partial t} \, dx dt
+2\nu \int_0^T \int_{\R^2} D(u): D(\varphi^{\eta}) \, dx dt
\end{multline*}
and
\begin{equation*}
 \int_{\R^2} u^{\varepsilon}_0(x) \cdot \varphi^{\eta}(0,x) \, dx
 \to \int_{\R^2} u_0(x) \cdot \varphi^{\eta}(0,x) \, dx.
\end{equation*}

For the non-linear term, we decompose in $(t_j, t_{j+1})\times \mathcal{O}_j$ as
$$
u^{\varepsilon}\otimes u^{\varepsilon} = \mathbb{P}_{\mathcal{O}_j} u^{\varepsilon} \otimes u^{\varepsilon} 
+\nabla q^{\varepsilon} \otimes \mathbb{P}_{\mathcal{O}_j} u^{\varepsilon}
+\nabla q^{\varepsilon} \otimes \nabla q^{\varepsilon}.
$$
Let us note that for any harmonic $\tilde q$ (i.e. $\Delta \tilde q=0$), we have the following relation:
\begin{equation}\label{algebra}
\int_{\mathcal{O}_j} ( \nabla\tilde q \otimes \nabla \tilde q) : \nabla \varphi^{\eta} 
= - \int_{\mathcal{O}_j} \div(\nabla \tilde q \otimes\nabla \tilde q) \cdot \varphi^{\eta} 
= -\int_{\mathcal{O}_j} \left( \frac{1}{2} \nabla |\nabla \tilde q |^2 \cdot \varphi^{\eta} + \Delta \tilde q \nabla \tilde q \cdot \varphi^{\eta} \right)= 0,
\end{equation}
because $\varphi^{\eta}$ is divergence free and compactly supported in $\mathcal{O}_j$.
From \eqref{fsr7.3}-\eqref{fsr7.1}, we have
$$
\int_{(t_j, t_{j+1})\times \mathcal{O}_j} (\mathbb{P}_{\mathcal{O}_j} u^{\varepsilon} \otimes u^{\varepsilon}) : \nabla\varphi^{\eta} \, dx dt
\to \int_{(t_j, t_{j+1})\times \mathcal{O}_j} (\mathbb{P}_{\mathcal{O}_j} u \otimes u) : \nabla\varphi^{\eta} \, dx dt,
$$
and
$$
\int_{(t_j, t_{j+1})\times \mathcal{O}_j} (\nabla q^{\varepsilon} \otimes \mathbb{P}_{\mathcal{O}_j} u^{\varepsilon}) : \nabla\varphi^{\eta} \, dx dt
\to \int_{(t_j, t_{j+1})\times \mathcal{O}_j} (\nabla q \otimes \mathbb{P}_{\mathcal{O}_j} u) : \nabla\varphi^{\eta} \, dx dt.
$$
Gathering the two above convergences and \eqref{algebra} applied to $q^\e$ and to $q$, we conclude
\begin{equation*}
\int_{(t_j, t_{j+1})\times \mathcal{O}_j} (u^{\varepsilon}\otimes u^{\varepsilon}) : \nabla\varphi^{\eta} \, dx dt
\to
\int_{(t_j, t_{j+1})\times \mathcal{O}_j} (u \otimes u ) : \nabla\varphi^{\eta} \, dx dt.
\end{equation*}

This ends the proof of Theorem~\ref{thmeta}.
\end{proof}

To end the proof of Theorem~\ref{mainthm2}, it is sufficient to pass to the limit $\eta\to 0$, thanks to Proposition~\ref{lem:bogo}.

\begin{proof}[Proof of Theorem~\ref{mainthm2}]
Let $T>0$ and $\varphi \in C^\infty_c([0,T)\times \R^2)$ with $\div \varphi=0$ fixed, then we consider $(\varphi^\eta)_{\eta\leq \eta_{1}}$ which approximate $\varphi$ (see Proposition~\ref{lem:bogo}) and $u$ a weak limit of $u^\varepsilon$ (see \eqref{rev00}). Theorem~\ref{thmeta} states that the limit $u$ verifies for any $\eta$
$$
-\int_0^T \int_{\R^2} u \cdot \left(\frac{\partial \varphi^{\eta}}{\partial t} + (u \cdot \nabla)\varphi^{\eta}\right) \, dx 
+2\nu \int_0^T \int_{\R^2} D(u): D(\varphi^{\eta}) \, dx = \int_{\R^2} u_0(x) \cdot \varphi^{\eta}(0,x) \, dx.
$$
As $u$ belongs to $L^\infty(0,T;L^2(\R^2))\cap L^2(0,T;H^1(\R^2))$, we deduce from the convergences \eqref{fsr4.4}-\eqref{fsr4.6} of $\varphi^\eta$ to $\varphi$ that
$$
-\int_0^T \int_{\R^2} u \cdot \left(\frac{\partial \varphi}{\partial t} + (u \cdot \nabla)\varphi \right) \, dx 
+2\nu \int_0^T \int_{\R^2} D(u): D(\varphi) \, dx = \int_{\R^2} u_0(x) \cdot \varphi(0,x) \, dx.
$$
By density, this equality is also true for any $\varphi\in C^1_c([0,T);\mathcal{V}(\R^2))$.
Noting that 
$$\int D(u):D(\varphi)=\frac12 \int \nabla u: \nabla \varphi + \frac12 \int \div u \div \varphi =\frac12 \int \nabla u: \nabla \varphi ,$$ 
we conclude that $u$ is a weak solution to the Navier-Stokes equations in $\R^2$ associated to $u_0$. By uniqueness of such a solution, we note that the weak convergence \eqref{rev00} holds for all sequence $(\varepsilon_n)$ converging to 0:
$$
u^{\varepsilon_n} \overset{*}{\rightharpoonup} u \quad \quad \text{in} \quad L^\infty(0,T;L^2(\R^2))\cap L^2(0,T;H^1(\R^2))
$$
as $n\to \infty$, without extracting a subsequence.

This ends the proof of Theorem~\ref{mainthm2}.
\end{proof}

\section{Final remarks and comments}\label{sec:final}

\subsection{The three dimensional case}

The main obstruction to the generalization to the three dimensional case is that the optimal decay estimates of the Stokes semigroup are not established for all time $t>0$.

For a solid with any shape moving in a three dimensional viscous fluid, Geissert, G\"otze and Hieber \cite[Theorem 4.1]{GGH} show the maximal regularity for the Stokes semigroup, locally in time. By an extension operator and Sobolev's embedding, we can deduce from their result the optimal $L^p-L^q$ estimates for some $p,q$ (see \cite[Proposition 3.1]{GHHS}). Even if we have to check that the $p,q$ reached are enough to perform the fixed point argument (as in Proposition~\ref{P01}), the problem is that these estimates are local in time, i.e. valid up to a time $T$, with the constants depending on $T$. Therefore, after the scaling argument (see the proof of Theorem~\ref{theo Stokes}), these estimates are independent of $\varepsilon$ only up to a time $T_{\varepsilon}=\varepsilon T$. Indeed, by the scaling property of the Navier-Stokes equations, the small obstacle problem is similar to the long-time behavior.

In dimension two, the global optimal $L^p-L^q$ decay estimates for the Stokes semigroup were obtained by Ervedoza, Hillairet and Lacave in \cite{EHL} when the solid is a disk in the whole plane. We guess that their analysis can be adapted in the exterior of a ball in dimension three, but it would require a considerable work, decomposing the Stokes equations on spherical coordinate system instead to polar decomposition, exhibiting the ``good unknown'' (see \cite[Proposition 2.3]{EHL} for the two dimensional case), and adapting the elliptic lemmas.

Finally, one should also adapt the fixed point argument performed in Proposition~\ref{P01}, which should not be too difficult, because the original proof of Kato \cite{Kato} holds for any dimension $n\geq 2$.

\subsection{Extension to massive pointwise particles}\label{sec:massive}

If some solids 
\begin{equation}\label{Si0}
\mathcal{S}^\varepsilon_{i,0} :=h_{i,0} + \varepsilon \mathcal{S}_{i,0},
\end{equation}
shrink to {\it massive pointwise} particle, i.e.
\begin{equation}\label{eq:mass}
 m^\varepsilon_{i} = m^1_{i}>0 \quad \text{and} \quad J^\varepsilon_{i}= \varepsilon^2 J^1_{i}>0,
\end{equation}
then the energy estimate gives directly the uniform estimates of the solid velocities. Therefore, we do not need here the analysis developed in Sections~\ref{sec:411}-\ref{estim_ell}, and we can prove the following result for any shape, and with several solids.

\begin{Theorem}
Let $N$ rigid bodies $\mathcal{S}^\varepsilon_i(t)$ of shape $\mathcal{S}_{i,0}$ \eqref{Si0} (with $\mathcal{S}_{i,0}$ smooth simply-connected compact subset of $\R^2$, with nonempty interior and where the center of mass of $\mathcal{S}_{i,0}$ is 0) inside a bounded domain $\Omega$, and where the positions $h_{i,0}\in \Omega$ and the size $\varepsilon$ are chosen such that
\begin{equation*}
 \mathcal{S}^\varepsilon_{i,0}\cap \mathcal{S}^\varepsilon_{j,0}= \emptyset \quad (i\neq j) \quad \text{and }\mathcal{S}^\varepsilon_{i,0} \subset \Omega.
\end{equation*}
Let $u^{\varepsilon}_0 \in \overline{\mathcal{V}_R(\mathcal{F}^{\varepsilon}_0)}^{L^2}$ such that
\begin{equation*}
u_0^{\varepsilon} \rightharpoonup u_0 \quad \text{in} \quad L^2(\Omega),
\end{equation*}
\begin{equation}\label{fsr8.4}
|\ell_{i,0}^{\varepsilon}|\leq C, \quad \varepsilon |r_{i,0}^{\varepsilon}|\leq C,\quad \forall i\in \{1,\dots,N\}.
\end{equation}
There exists a global weak solution $(u^{\varepsilon},\ h_i^\varepsilon, \ \theta_i^\varepsilon)$ see \cite{SanStaTuc2002a} (see also \cite{DesEst1999a}, \cite{GunLeeSer2000a}).
Then there exists $T>0$ such that
\begin{equation*}
u^{\varepsilon} \overset{*}{\rightharpoonup} u \quad \quad \text{ in } \quad L^\infty(0,T;L^2(\Omega))\cap L^2(0,T;H^1_{0}(\Omega)) 
\end{equation*}
where $u$ is the weak solution of the Navier-Stokes equations associated to $u_0$ in $\Omega$.
\end{Theorem}

\begin{Remark}
To adapt Proposition~\ref{lem:bogo}, we need the existence of a positive distance between the rigid bodies, independent of $\varepsilon$. Therefore, denoting by $h_{i}$ the limit of $h_{i}^\varepsilon$, the time $T$ in the above theorem corresponds to a time such that 
$$
| h_i(t)-h_j(t)| > 0 \quad (i\neq j, \quad t\in [0,T]).
$$
The existence of such a $T$ is ensured by \eqref{fsr8.4}. 
\end{Remark}
Namely, we can prove the following:
\begin{Proposition}
Let $T>0$, $\varphi \in C^\infty_c([0,T)\times \Omega)$ with $\div \varphi=0$ and consider $\eta_{1}>0$ such that
\begin{equation*}
 | h_{i}(t)-h_j(t) | \geq 2\eta_{1} \quad \text{for all } t \in [0,T] \text{ and } i \neq j
\end{equation*}
and
\begin{equation*}
\dist(\supp \varphi(t,\cdot), \partial \Omega) \geq 2\eta_{1} \quad \text{for all } t \in [0,T] .
\end{equation*}
For any $\eta\leq \eta_{1}$ there exists $\varphi^\eta \in W^{1,\infty}_c([0,T);H^1_0(\Omega))$ satisfying
\begin{equation*}
\div \varphi^\eta =0 \quad \text{in} \ [0,T)\times \Omega,
\end{equation*}
\begin{equation*}
\varphi^\eta \equiv 0 \quad t\in (0,T), \quad x\in B\left(h_i(t),\frac{\eta}{2}\right),
\end{equation*}
\begin{equation*}
 \varphi^\eta \overset{*}{\rightharpoonup} \varphi \quad L^\infty(0,T;H^1(\Omega)),
\end{equation*}
\begin{equation*}
 \partial_t \varphi^\eta \overset{*}{\rightharpoonup} \partial_t \varphi \quad L^\infty(0,T;L^2(\Omega)).
\end{equation*}
\end{Proposition}

 In contrast, we do not need a positive distance between the bodies and the exterior boundary $\partial\Omega$. In particular, in the case of a single rigid body (i.e. $N=1$), we can take $T$ arbitrary large.

The rest of the proof can be done following Section~\ref{sec:limit}.

\bigskip

\noindent
{\bf Acknowledgements.} We want to warmly thank Isabelle Gallagher, David G\'erard-Varet, Matthieu Hillairet and Franck Sueur for several fruitful discussions.

We are partially supported by the Agence Nationale de la Recherche, Project IFSMACS, grant ANR-15-CE40-0010.
C.L. is partially supported by the Agence Nationale de la Recherche, Project DYFICOLTI, grant ANR-13-BS01-0003-01, and by the project \emph{Instabilities in Hydrodynamics} funded by the Paris city hall (program \emph{Emergences}) and the Fondation Sciences Math\'ematiques de Paris.

The authors are also grateful to the anonymous referees for their valuable comments on the first version of this article.


\adrese

\end{document}